\documentclass[aap]{imsart}

\RequirePackage{amsthm,amsmath,amsfonts,amssymb}
\RequirePackage[numbers,sort&compress]{natbib}
\RequirePackage[colorlinks,citecolor=blue,urlcolor=blue]{hyperref}

\usepackage{bigints}
\usepackage{mathrsfs}
\usepackage{commath}
\usepackage{tabularx}

\usepackage{xcolor}
\usepackage{graphicx}
\graphicspath{{./images/}}
\usepackage{bbm}
\usepackage{bm}

\startlocaldefs
\theoremstyle{plain}

\newtheorem{theorem}{Theorem}[section]
\newtheorem{lemma}[theorem]{Lemma}
\newtheorem{proposition}{Proposition}[section]
\theoremstyle{remark}
\newtheorem{definition}[theorem]{Definition}

\newtheorem{remark}{Remark}[section]

\def\ms{\mathsf}

\newcommand{\N}{\mathbb{N}}

\newcommand{\R}{\mathbb{R}}

\renewcommand{\Pr}{\mathbb{P}}
\newcommand{\E}{\mathbb{E}}

\renewcommand{\d}{\mathrm{d}}
\newcommand{\1}{\mathbbm{1}}

\DeclareMathOperator{\V}{Var}
\DeclareMathOperator{\Cov}{Cov}

\newcommand{\e}{\epsilon}

\newcommand{\lp}{\left(}
\newcommand{\rp}{\right)}
\newcommand{\lb}{\left[}
\newcommand{\rb}{\right]}

\newcommand{\Pn}{\mathcal P_n}

\newcommand{\Y}{\mathcal Y}
\newcommand{\consisest}{U_k(m^k)}

\newcommand{\bx}{{\bf x}}
\newcommand{\by}{{\bf y}}
\newcommand{\bz}{{\bf z}}
\newcommand{\vep}{\varepsilon}
\newcommand{\Rk}{R_k(C_kn/m)}
\newcommand{\mE}{\mathcal E}

\endlocaldefs

\begin{document}

\begin{frontmatter}
\title{Layered Hill estimator for extreme data in clusters}
\runtitle{Layered Hill estimator}

\begin{aug}

\author[A]{\fnms{Taegyu}~\snm{Kang}\ead[label=e1]{tkang81@gatech.edu}},
\author[B]{\fnms{Takashi}~\snm{Owada}\thanks{[\textbf{Corresponding author}]}\ead[label=e2]{owada@purdue.edu}}


\address[A]{H. Milton Stewart School of Industrial and Systems Engineering, Georgia Institute of Technology \printead[presep={ ,\ }]{e1}}

\address[B]{Department of Statistics,
Purdue University \printead[presep={ ,\ }]{e2}}

\end{aug}

\begin{abstract}
A new estimator is proposed for estimating the tail exponent of a heavy-tailed distribution. This estimator, referred to as the layered Hill estimator, is a generalization of the traditional Hill estimator, building upon a layered structure formed by clusters of extreme values. We argue that the layered Hill estimator provides a robust alternative to the traditional approach, exhibiting desirable asymptotic properties such as consistency and asymptotic normality for the tail exponent. Both theoretical analysis and simulation studies demonstrate that the layered Hill estimator shows significantly better and more robust performance, particularly when a portion of the extreme data is missing.
\end{abstract}



\end{frontmatter}

\section{Introduction}
\label{sec: introduction}

One of the primary challenges in analyzing data from a heavy-tailed distribution is estimating the thickness of the tail. A heavy-tailed distribution is typically modeled by a tail that decays according to a power law, with the decay rate determined by the so-called tail exponent. 
In Extreme Value Theory, the most widely used estimator for the tail exponent is the \emph{Hill estimator}, which was introduced by \cite{Hill:1975}, with detailed expositions   in \cite{embrechts:kluppelberg:mikosch:1997, dehaan:ferreira:2006, Resnick:2007}.  The Hill estimator is defined as
\begin{equation}  \label{e:Hill.intro}
H_{m,n} := \frac{1}{m} \sum_{i=1}^m \log \frac{X_{(i)}}{X_{(m+1)}}, 
\end{equation}
where $X_{(1)} \ge X_{(2)} \ge \cdots \ge X_{(n)}$ are the order statistics from a sequence of random variables $(X_i)_{i=1}^n$, and the cut-off sequence $m = m(n)$ satisfies $m \to \infty$ and $m/n \to 0$ as $n \to \infty$. The theoretical properties of the Hill estimator, including its consistency and asymptotic normality for the tail exponent, have been extensively studied under appropriate conditions (see, e.g., \cite[Chapters 4 and 9]{Resnick:2007}).

Despite these theoretical guarantees, the Hill estimator’s behavior can be sensitive, often limiting its practical usefulness. One such instance occurs when some extreme data are missing. Missing extreme values often arise in datasets in natural disasters \cite{burroughs:tebbens:2001, burroughs:tebbens:2002}. From a theoretical perspective, \cite{beirlant:alves:gomes:2016} and \cite{chakrabarty:samorodnitsky:2012} examined the properties of heavy-tailed distributions when the tail is truncated, and proposed methods for estimating the tail exponent in such circumstances. A key issue when extreme values are missing is that the Hill estimator tends to underestimate the tail exponent due to the absence of large observations. To illustrate this, Table \ref{table:intro.missing} presents a simple simulation. As more extreme values are missing, the Hill estimator increases, leading to an underestimation of the thickness of the tail.

\begin{table}[t]
\centering
\begin{tabular}{l | c c c c c}
\hline \hline
Missing rate & $0\%$ & $25\%$ & $50\%$ & $75\%$ & $100\%$ \\
\hline
Point estimate & $2.503$ & $2.840$ & $3.114$ & $3.377$ & $3.627$ \\
\hline \hline
\end{tabular}
\vspace{5pt}
\caption{\footnotesize Point estimates obtained by the Hill estimators for varying proportions of missing extremes: We generate $n = 10,000$ random points from the density $f(x) = C |x|^{-2.5} \1\{|x| \geq 1\}$ in $\R^2$. The Hill estimators are computed using the $m = m(n) = n^{0.5}$ most extreme points, determined by their Euclidean distance from the origin. Each column corresponds to the case that a specific percentage of the most extreme points are removed. For instance, the second column gives an estimate when $0.25 m$ of the most extreme points are removed, and the Hill estimator is calculated using the $0.25m$ to $1.25m$ largest extreme points. }
\label{table:intro.missing}
\end{table}

The Hill plot is a two-dimensional graph where the cut-off sequence $m$ is plotted along the $x$-axis, and the corresponding values of $H_{m,n}$ are plotted along the $y$-axis. One fundamental issue in the presence of missing extremes is that the Hill plot tends to become an increasing function of the cut-off sequence without any flat regions \cite{Zou;Davis;Samorodnitsky:2020, Xu;Davis;Samorodnitsky:2022}. This phenomenon makes it particularly difficult to determine an appropriate cut-off sequence for estimating the tail exponent by the Hill estimator. To resolve this issue, \cite{Zou;Davis;Samorodnitsky:2020} introduced the ``Hill Estimator Without Extremes" (HEWE) process and proposed a statistical algorithm to estimate both the missing rate of extremes and the underlying tail exponent. 

In this paper, we approach this problem using a different methodology. Specifically, we make use of the layered structure of clusters formed by extreme observations. The study of such a layered structure was first introduced in the pioneering paper \cite{Adler;Bobrowski;Weinberger:2014} in the context of manifold learning in Topological Data Analysis. Since then, this structure has been extensively examined in relation to the behavior of various topological invariants \cite{Owada:2018, owada:bobrowski:2020, Thomas:2023}. The primary aim of this paper is to propose a general version of the Hill estimators by exploiting such layered structure of extremes. This new estimator is expected to be significantly more robust, particularly in the presence of missing extreme data.

To provide a rough understanding of the main idea, let us consider a specific situation. First, we generate a set of random points $\mathcal{X}_n = \{ X_1, \dots, X_n \} \subset \R^d$ sampled from a common spherically symmetric density with a heavy tail.
Next, we form a geometric graph $G(\mathcal{X}_n; 1)$ on the vertex set $\mathcal{X}_n$ using a unit connectivity radius. Specifically, an edge $\{ X_i, X_j \}$ is added whenever $|X_i - X_j| \leq 1$, where $|\cdot|$ denotes the Euclidean norm in $\R^d$. For a positive integer $K\ge2$, we fix a feasible and connected graph $\Gamma_k$ on $k$ vertices for each $k = 2, \dots, K$. In particular, $\Gamma_2$ necessarily represents an edge, while the choice of graphs for $k \ge 3$ is arbitrary.

Following the arguments in \cite{Owada;Adler:2017} and \cite{Owada:2017}, 
we partition the space $\R^d$ into several distinct layers, as illustrated in Figure \ref{fig:layer}. 
We then construct the ``layered" Hill estimators, corresponding to each of the layers in Figure \ref{fig:layer}. 
In what follows, we refer to the original Hill estimator in \cite{Hill:1975} as the traditional Hill estimator to differentiate it from the newly proposed ones.

\begin{figure}[t]
\centering
\includegraphics[height = 3.5cm, width = 13.5cm]{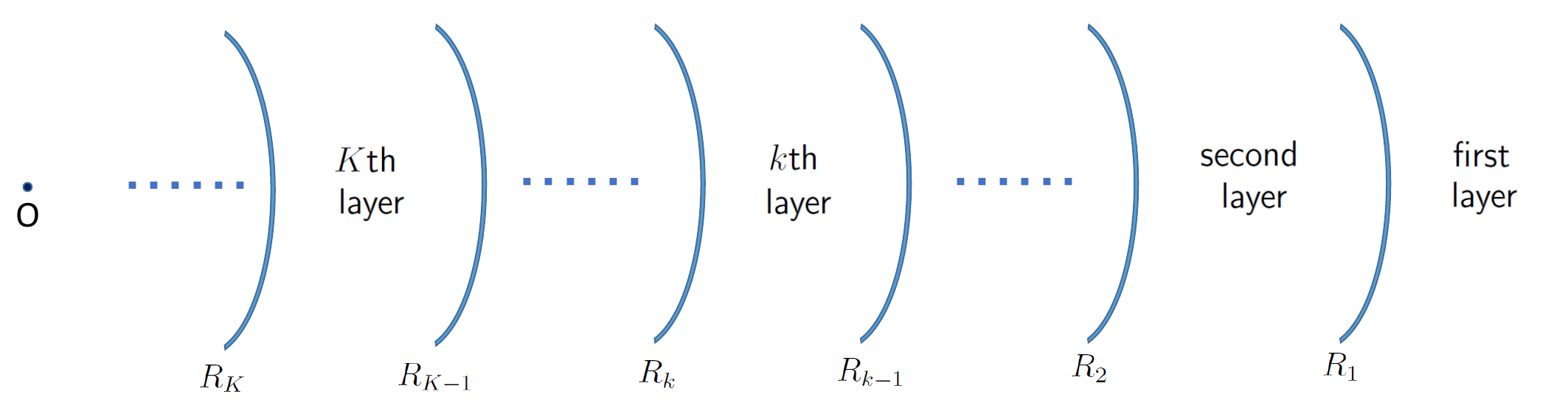}
\caption{Layered structure of clusters of extremes. Here, $R_k$ is a sequence of $n/m$ that diverges as $n \to \infty$, where $m=m(n)$ is the cut-off defined in \eqref{e:Hill.intro} (see \eqref{e:def.radius.function.R} for a formal definition of $R_k$). The ``$k$th layer” refers to the annulus with inner radius $R_k$ and outer radius $R_{k-1}$. }
\label{fig:layer}
\end{figure}

Specifically, we construct the first layered Hill estimator by using the $m$ largest values in the set $\{|X_i|\}_{i=1}^n$. This estimator is equivalent to the traditional Hill estimator, relying only on the random points in the first layer of Figure \ref{fig:layer}, without incorporating points from other layers. 
Next, the second layered Hill estimator is constructed from the $m^2$ largest values in the set 
$$
\Big\{ \min \big( |X_i|, |X_j| \big): |X_i - X_j| \leq 1, \, 1 \leq i < j \leq n \Big\},
$$
which represents a collection of the values of $\min\big( |X_i|, |X_j| \big)$ where $\{ X_i, X_j \}$ forms an edge in the random geometric graph.  
In this case, the random points corresponding to these $m^2$ largest values are asymptotically distributed in the second layer of Figure \ref{fig:layer}. Indeed, the first layer lies farther away from the origin than the second layer, so we do not asymptotically observe pairs of extremes forming edges within the first layer.

More generally, for each $k = 2, \dots, K$, the $k$th layered Hill estimator is constructed from the $m^k$ largest values in the set
$$
\Big\{ \min \big( |X_{i_1}|, \dots, |X_{i_k}| \big): G\big( \{ X_{i_1}, \dots, X_{i_k} \}; 1 \big) \cong \Gamma_k, \ 1 \leq i_1 < \cdots < i_k \leq n \Big\},
$$
representing a collection of the values of $\min \big( |X_{i_1}|, \dots, |X_{i_k}| \big)$, such that the geometric graph induced on $(X_{i_1}, \dots, X_{i_k})$ becomes isomorphic to $\Gamma_k$. Then, all the random points used for this estimator are asymptotically drawn from the $k$th layer in Figure \ref{fig:layer}. Since $\Gamma_k$ is a  \emph{connected} graph, these selected $k$ extreme points must be “close” to one another. Such proximity does not occur within the first through $(k-1)$st layers, as these layers are farther from the origin compared to the $k$th layer.

Now, let us consider the case where a portion of  the extreme random points is missing. Specifically, we remove some of the extremes from the first layer in Figure \ref{fig:layer}. Then, the traditional Hill estimator, which is equivalent to the first layered Hill estimator, may significantly underestimate the tail exponent, as shown in Table \ref{table:intro.missing}. In contrast, the  $k$th layered Hill estimator for $k \geq 2$ is less affected by the absence of extremes in the first layer, as it does not rely on points from that layer. 
Even when some of the extremes are missing from the first layer, information about the tail exponent is still encoded in the random points within the higher-order layers, and, the $k$th layered Hill estimators for $k \geq 2$ are expected to outperform the traditional Hill estimator. In practical applications, when the possibility of missing extremes is suspected, it would be sensible to consider a mixture of these layered estimators to reduce the impact of missing data.

The remainder of the paper is structured as follows. In Section \ref{sec:formal.def}, we formally introduce the layered Hill estimators. Section \ref{sec:consistency} is devoted to establishing the consistency of these estimators. As an intermediate step, we verify that the layered tail empirical measure converges, in probability, to a deterministic Radon measure under suitable scaling. In Section \ref{sec:asym.normality}, we prove the asymptotic normality of the layered Hill estimator. This requires several intermediate steps, including the asymptotic normality of both the layered tail empirical measure and a certain stochastic process induced by this measure.
Section \ref{sec:non.random.centering} addresses the asymptotic normality of the layered Hill estimator under non-random centering, subject to a set of technical conditions. The overall proof strategy involves a set of new methods that have not appeared in the proof for the traditional Hill estimator. For example, when $k \ge 2$, the $k$th layered Hill estimator counts $k$-tuples satisfying certain geometric constraints, such as the formation of subgraphs isomorphic to $\Gamma_k$. Due to such geometric constraints, previously developed techniques for proving asymptotic normality of the traditional Hill estimator are not applicable. Instead, we need to apply Stein's method for normal approximation (\cite[Theorem 2.4]{Penrose:2003}), in conjunction with the multivariate Mecke formula (\cite[Theorem 4.4]{Last;Penrose:2017}).
Finally, in Section \ref{sec:simulation}, we present simulation studies to compare the performance of the traditional Hill estimator and the layered Hill estimators when a portion of the extremes is missing. All proofs are deferred to Section \ref{sec:proofs}.

Before concluding the Introduction, we must highlight some limitations of our method. First, as mentioned in the previous paragraph, one of the key components in our proof is Stein's method for normal approximation. However, for the application of this result, our calculations require to assume that a probability density exists, and regular variation condition is imposed directly on the tail of the density functions. This assumption is rather restrictive compared to the standard approach, where regular variation is assumed on the tail of distribution functions. Nonetheless, such a strict condition is important for our method. 

Second, throughout the paper, we assume that the probability density is spherically symmetric. While this assumption simplifies our analysis, it is not essential; all the results in this paper can be extended to densities with more general level sets. For instance, \cite{Thomas:2023} examined the layered structure generated by a density whose level sets form concentric convex sets. From this viewpoint, combining our  method with existing methods for estimating the shape of level sets of a density function (see \cite{Berthet;Einmahl:2022}) seems to be an interesting direction for future research. 

Finally, we assume  that the random sample is drawn from a (inhomogeneous) Poisson point process. The main reason for this assumption is that the spatial independence of a Poisson point process simplifies the analysis. However, we believe that all of our  results can be extended to the case where the random sample is drawn from a binomial process. Extending the results in this manner requires the technique known as de-Poissonization (see, e.g., \cite[Section 2.5]{Penrose:2003}).  In particular, the methods  in \cite{Owada;Adler:2017} and \cite{Wei;Owada:2022}  will be useful for the required de-Poissonization.

Throughout the paper, we use the following standard notation. First, $|x|$ denotes the Euclidean norm of a vector $x \in \R^d$, but when we are given a set $A$, $|A|$ also stands for the cardinality of $A$. Moreover, let $\mathbb{S}^{d-1}$ represent the $(d-1)$-dimensional unit sphere in $\R^d$ centered at the origin, $s_{d-1}$ be the surface area of $\mathbb{S}^{d-1}$, and $\kappa_d$ denote the volume of the $d$-dimensional unit ball in $\R^d$. Additionally, for any point $y$ in a given space, $\delta_y(\cdot)$ denotes the Dirac measure at the point $y$. Furthermore, $\Rightarrow$ denotes weak convergence, and $\stackrel{\Pr}{\to}$ means convergence in probability. Finally, given a sequence $(X_n)_{n \ge 1}$ of random variables, we write $X_n = o_p(1)$ if $X_n\stackrel{\Pr}{\to}0$ as $n\to\infty$. 
\medskip

\section{Layered Hill estimators}  \label{sec:formal.def}

In this section, we formally set up the layered Hill estimator. We assume a spherically symmetric probability density function $f: \R^d \to [0,\infty)$, $d\ge1$, with a  regularly varying tail: there exists a tail exponent $\alpha >d$, such that for every  (equivalently, some, due to spherical symmetry of $f$) $\theta \in \mathbb{S}^{d-1}$ and $r\in (0,\infty)$, it holds that 
\begin{equation}  \label{e:RV.f}
\lim_{t\to\infty} \frac{f(tr\theta)}{f(t\theta)}=r^{-\alpha}. 
\end{equation}
Because of spherical symmetry, we can define $f(r):= f(r\theta)$ for any $r\ge0$ and $\theta \in \mathbb{S}^{d-1}$.  Moreover, \eqref{e:RV.f} implies that $f(r)=r^{-\alpha}L(r)$ for some   slowly varying function $L$. We assume that $L$ is eventually monotone,  i.e., monotone on $(r_0,\infty)$ for some large enough $r_0>0$. In addition,  we assume that $L$ is differentiable   on $(r_0,\infty)$ and satisfies, for all $r\ge r_0$,
\begin{equation}  \label{e:smoothness.L}
\frac{rL'(r)}{L(r) }< \alpha -d.
\end{equation}
Condition \eqref{e:smoothness.L}  is a standard smoothness assumption on $L$; it holds, for instance,  when $L$ is differentiable and belongs to the Zygmund class (see Section 1.5 in \cite{Bingham;Goldie;Teugels:1987}).

For every $n\ge1$, we define $\Pn$ as a (inhomogeneous) Poisson point process in $\R^d$ whose intensity is given by $nf$. 
 Let $m = m(n)$ be a sequence of positive integers defined by 
\begin{equation}  \label{e:RV.m(n)}
m(n) =[ n^\beta ], \ \ \ \beta\in (0,1), 
\end{equation}
where $[\cdot]$ denotes an integer part. Clearly, $m=m(n)$  satisfies
\begin{equation}  \label{e:cond.m(n)}
m(n) \to \infty \ \ \text{and} \ \ m(n)/n \to 0 \ \ \text{as} \ n \to \infty.
\end{equation}

In this setting, Lemma  \ref{l:no.oscillation} in the Appendix guarantees that, for each $k\ge1$,  there exists an increasing function $R_k:[0,\infty)\to[0,\infty)$, diverging to infinity, such that
\begin{equation}  \label{e:def.radius.function.R}
t^k R_k(t)^d f\big( R_k(t)\big)^k \to \alpha k -d, \ \ \text{as } t\to \infty. 
\end{equation}

Given $k\ge1$, we define $h_1(x)\equiv 1$ for all $x\in \R^d$, and 
for  $k\ge2$, let $h_k: (\R^d)^k \to \{0,1\}$ be an indicator function which is used to impose a geometric constraint for our layered Hill estimators. Specifically, we put the following three conditions on $h_k$. 
\vspace{5pt}

\begin{itemize}
\item[$(i)$] $h_k$ is permutation invariant, i.e., 
$$
h_k (y_{\sigma(1)}, \dots, y_{\sigma(k)}) = h_k (y_1,\dots,y_k), \ \ \ y_i \in \R^d
$$
for every permutation $\sigma$ on $\{ 1,\dots,k \}$. 
\item[$(ii)$] $h_k$ is translation invariant: 
$$
h_k (x+y_1,\dots,x+y_k) = h_k (y_1,\dots,y_k), \ \ \ x, y_i \in \R^d. 
$$
\item[$(iii)$] There exists a constant $L\in (0,\infty)$, such that 
\begin{equation}  \label{e:denseness.hk}
h_k(y_1,\dots,y_k)=0, \ \ \text{whenever } \ms{diam}(y_1,\dots,y_k) >L, 
\end{equation}
where $\ms{diam}(y_1,\dots,y_k)$ denotes the diameter of a point set $\{ y_1,\dots, y_k \}$ in terms of the Euclidean distance. 
\end{itemize}
Since $h_k$ is permutation invariant, we often use the notation $h_k (\Y)$ for a $k$-point set $\Y \subset \R^d$, to represent $h_k(y_1,\dots,y_k)$ for any permutation $\{  y_1,\dots,y_k\}$ of the elements in $\Y$. 
There are several typical examples of such indicator functions: for each $k\ge2$, 
$$
h_k(y_1,\dots,y_k) = \1 \big\{ \ms{diam}(y_1,\dots,y_k) \le t \big\}
$$
for some $t>0$, and 
$$
h_k(y_1,\dots,y_k) = \1 \big\{  G(\{ y_1,\dots,y_k \}; t) \cong\Gamma_k \big\}, 
$$
where 
$\cong$ means graph isomorphism and $\Gamma_k$ represents a fixed, feasible, and connected graph of $k$ vertices. One can readily check that these indicators are permutation and translation invariant. In particular, the connectivity of $\Gamma_k$ ensures condition (iii) of the latter indicator.

For each $k\ge1$, we define a certain point process on $E_k :=  \big( \overline{\R^d} \setminus \{ \textbf{0} \} \big)^k$. Here, $\overline{\R^d}$ denotes a compactification of $\R^d$ obtained by adding  points at infinity,  and $\overline{\R^d} \setminus \{\bf 0\}$ is its one-point uncompactification  obtained by removing the origin $\bf 0$; see~\cite[Section 6.1.3]{Resnick:2007} for a detailed discussion of their role in extreme value theory. The space $E_k$ encodes the spatial information on the $k$th layered structure of extremes generated by $\Pn$. More precisely, the point process defined below counts the number of $k$-point subsets of $\Pn$ that satisfy a geometric constraint specified by $h_k$, and is called the layered tail empirical measure. 

\begin{definition}[Layered tail empirical measure]
\label{def:layered.tail.empirical.measures}
Given $k \ge1$ and $m=m(n)$ satisfying \eqref{e:cond.m(n)}, we define 
$$
\nu_{k,n}(\cdot) := \frac{1}{k!} \sum_{\bm{y} \in (\Pn)_{\neq}^k} h_k(\bm{y}) \, \delta_{\bm{y}/ R_k(C_k n/m)} (\cdot), \ \ \ n\ge1, 
$$
where 
$
(\Pn)_{\neq}^k := \big\{ (y_1,\dots,y_k) \in \Pn^k: y_i \neq y_j \text{ for } i \neq j \big\}
$
is the $k$th factorial measure of $\Pn$ and 
\begin{equation}  \label{e:def.Ck}
C_k := \Big(  \frac{s_{d-1}}{k!} \int_{(\R^d)^{k-1}}h_k(0,z_1,\dots,z_{k-1})\, \d \bz\Big)^{1/k}. 
\end{equation}
\end{definition}

The layered Hill estimator can be constructed from the layered tail empirical measure $\nu_{k,n}$. A  challenge is, however, that $\nu_{k,n}$ involves the radius $R_k(C_k n/m)$, whose exact value is unknown. To resolve this issue, we need to propose a consistent estimator for $R_k(C_k n/m)$. To this end, for each $k\ge1$ and $n\ge1$, we set $U_k(1) \ge U_k(2) \ge \cdots \ge U_k\lp G_{k,n} \rp$  with $G_{k,n}:= \binom{|\Pn|}{k}$, as the order statistics constructed from the values in 
$$
\big\{ h_k(\Y) \ms{min}(\Y): \Y\subset \Pn, \, |\Y|=k \big\}, 
$$
where $\ms{min}(\Y) =\ms{min}(y_1,\dots,y_k):= \min_{1\le i \le k}|y_i|$.  
Proposition \ref{p:consistency.tail.measure} in the next section  verifies, formally, that the $m^k$th largest order statistic, $\consisest$, serves as a consistent estimator for the radius $R_k(C_k n/m)$. Based on this estimator, we define the corresponding layered tail empirical measure by 
\begin{equation}  \label{e:est.tail.empirical}
\hat \nu_{k,n} (\cdot ):= \frac{1}{k!}\, \sum_{\bm{y}\in (\Pn)_{\neq}^k } h_k(\bm{y}) \, \delta_{\bm{y}/\consisest}(\cdot ), \ \ \ n\ge1, 
\end{equation}
which no longer contains unknown quantities. 
We set $\hat \nu_{k,n}$ to be a null measure whenever $\consisest = 0$. 

Given $k\ge1$, let 
\begin{equation}  \label{e:def.Br}
B_r = \big\{ (x_1,\dots,x_k)\in E_k: |x_i|\ge r  \text{ for all } i=1,\dots,k \big\}, \ \ \ r>0. 
\end{equation}
Now, the layered Hill estimator can be defined as follows. See also Figure \ref{fig:construction}. 

\begin{figure}[t]
\centering
\includegraphics[scale=0.35]{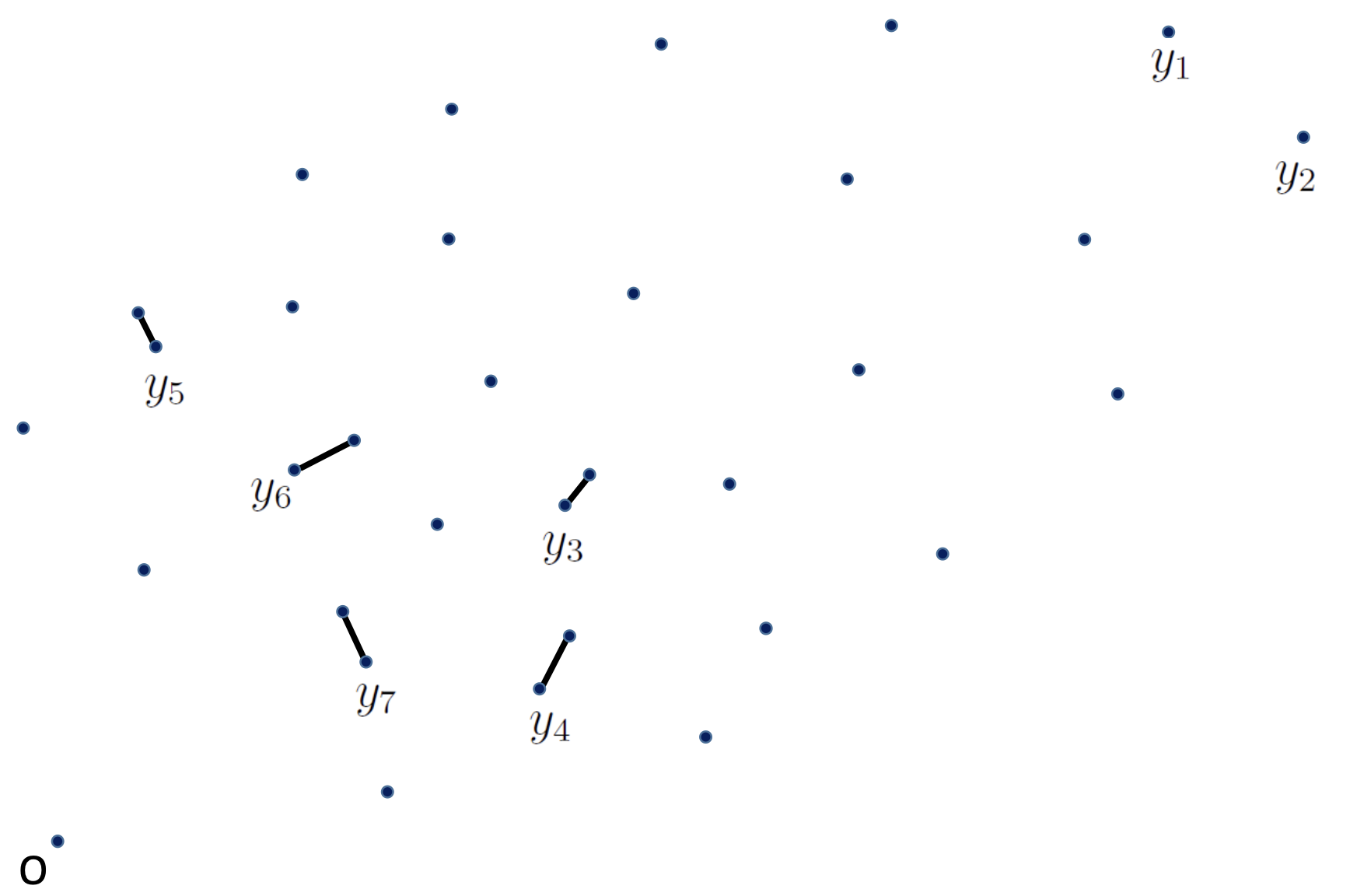}
\caption{\footnotesize{We set $n=32$, $m=2$, and $h_2(x_1,x_2)=\1 \big\{ |x_1-x_2| \le 1\big\}$. The first layered Hill estimator is calculated using $|y_1|$ and $|y_2|$, as they are the two largest extreme values, in terms of the distance from the origin. Meanwhile, the second layered Hill estimator is obtained from $|y_3|, |y_4|, |y_5|$, and $|y_6|$, which are the $m^2=4$ largest extreme values corresponding to the endpoints of edges.}}
\label{fig:construction}
\end{figure}

\begin{definition}[Layered Hill estimator] \label{def:layered.Hill.est}
For each $k \ge1$ and $m=m(n)$ satisfying \eqref{e:cond.m(n)},   define the $k$th layered Hill estimator by 
\begin{align*}
H_{k,m,n}  &:= \int_{1}^{\infty}m^{-k} \hat{\nu}_{k,n}(B_t) \frac{\d t}{t} 
= \frac{1}{m^k} \sum_{\Y\subset \Pn, |\Y|=k} h_k(\Y) \log  \bigg( \frac{\ms{min}(\Y)}{\consisest} \vee 1\bigg). 
\end{align*}
\end{definition}
Note that the sum $\sum_{\Y\subset \Pn, |\Y|=k}$ is taken over all $k$-tuples from $\Pn$, without considering any permutations of the selected points, in contrast to the summations in Definition \ref{def:layered.tail.empirical.measures} and \eqref{e:est.tail.empirical}. 
We also remark that when $d=k=1$, the layered Hill estimator of the first-order coincides with the traditional Hill estimator defined in  \cite[Chapter 4]{Resnick:2007}. 
\vspace{36pt}

\section{Asymptotic properties of the layered Hill estimators}  \label{sec:asym.properties}

\subsection{Consistency} \label{sec:consistency}

In this section, we establish the consistency of the layered Hill estimator in  Definition \ref{def:layered.Hill.est}. To achieve this, we first prove convergence in probability for the layered tail empirical measure in Definition \ref{def:layered.tail.empirical.measures}, along with the corresponding result for the measure in \eqref{e:est.tail.empirical}, in the space $M_+(E_k)$ of non-negative Radon measures on $E_k$. We endow $M_+(E_k)$ with the vague topology.

Given a measure $\nu$ on $E_k$ equipped with its Borel $\sigma$-algebra and a $\nu$-integrable function $\phi$, we often write $\nu(\phi)=\int \phi \d \nu$. For each $k\ge1$, define the measure $\mu_k\in M_+(E_k)$ by  the relation 
$$
\mu_k (\phi) = \frac{\alpha k-d}{s_{d-1}} \int_{\R^d} \phi (x,\dots,x) |x|^{-\alpha k} \d x, \ \ \ \phi \in C_c^+(E_k), 
$$
with $C_c^+(E_k)$ being the space of continuous and non-negative functions with compact support. 

\begin{proposition}   \label{p:consistency.tail.measure}
$(i)$ It holds that as $n\to\infty$, 
$$
m^{-k} \nu_{k,n} \stackrel{\Pr}{\to} \mu_k, \ \ \text{in }  M_+(E_k). 
$$
$(ii)$ We have, as $n\to\infty$, 
\begin{equation}  \label{e:consistent.est.by.order.stat}
\frac{\consisest}{R_k(C_kn/m)} \stackrel{\Pr}{\to} 1, 
\end{equation}
and also, 
$$
m^{-k} \hat{\nu}_{k,n} \overset{\Pr}{\to} \mu_k, \ \ \text{in }  M_+(E_k). 
$$
\end{proposition}

The theorem below ensures the consistency for the layered Hill estimators. 

\begin{theorem}
\label{t:consistency.Hill}
For each $k\ge1$, it holds that 
$$
H_{k,m,n} \stackrel{\Pr}{\to} \frac{1}{\alpha k-d}, \ \ \text{as } n\to\infty; 
$$
equivalently, 
\begin{equation}  \label{e:consis.alpha}
d/k +  (kH_{k,m,n})^{-1} \stackrel{\Pr}{\to} \alpha, \ \  \text{as } n\to\infty. 
\end{equation}
\end{theorem}
\medskip

\subsection{Asymptotic normality} \label{sec:asym.normality}
Subsequently, we verify the asymptotic normality of the layered Hill estimators. For this purpose, we need  a  few intermediate steps. First, we show that, with proper normalization, the layered tail empirical measure converges weakly to some Gaussian random measure. Second,   by using this result, the asymptotic normality of the  process
\begin{equation}  \label{e:nu.kn.proc.version}
\nu_{k,n} (t) := \nu_{k,n}(B_t) = \sum_{\Y \subset \Pn, \, |\Y|=k} h_k (\Y)\, \1 \big\{  \ms{min}(\Y)\ge t R_k (C_k n/m)\big\}, \ \ \ t>0, 
\end{equation}
is derived in the space $D(0,\infty]$ of right-continuous functions on $(0,\infty]$ with left limits. Finally, based on the functional asymptotic normality for the process \eqref{e:nu.kn.proc.version}, the desired asymptotic normality for the layered Hill estimator is deduced. 

For ease of notations,  we will  introduce the following sequences and functions. 
Since the limit of  $n f\big( R_k(C_k n/m) \big)$ exists in $[0,\infty]$ without oscillation (see Lemma \ref{l:no.oscillation} in the Appendix),  we can define the sequence $(\tau_{k,n})_{n\ge1}$ by
\begin{equation} \label{e:def.tau_kn} 
\begin{aligned}
\tau_{k,n} := \begin{cases}
m^k & \text{ if } n f \big(  R_k(C_k n/m)  \big) \to 0 \text{ or } \xi \in (0, \infty), \\[3pt]
m^k \big\{  n f \big( R_k(C_k n/m)  \big) \big\}^{k-1} & \text{ if } n f \big( R_k(C_k n/m)  \big) \to \infty.
\end{cases}
\end{aligned}
\end{equation}
Note that by condition \eqref{e:def.radius.function.R}, $n f \lp R_k (C_k n/m) \rp$ either converges to $0$, tends to a positive constant $\xi$, or diverges to $\infty$.   We will use \eqref{e:def.tau_kn} for an  appropriate scaling factor for the normalization. 
For  measurable functions $\phi_i$, $i=1,2$, we define
\begin{align}
\begin{split}  \label{e:def.Vk}
V_k(\phi_1, \phi_2) := \begin{cases}
\frac{(\alpha k-d) D_{k, k}}{k! (C_k)^k} \int_{\R^d} \phi_1(x, \dots, x)\phi_2(x, \dots, x) |x|^{-\alpha k} \d x \\[5pt]
&\hspace{-120pt} \text{ if } n f \big( R_k(C_kn/m)  \big) \to 0, \\[5pt]
\sum_{\ell = 1}^k  \frac{\xi^{k- \ell}(\alpha k-d) D_{k,\ell}}{\ell! ( (k-\ell)! )^2 (C_k)^k} \int_{\R^d} \phi_1(x, \dots, x)\phi_2(x, \dots, x)|x|^{-\alpha (2k - \ell)} \d x \\[5pt]
&\hspace{-120pt} \text{ if } n f \big( R_k(C_kn/m)  \big) \to \xi \in (0, \infty), \\[5pt]
\frac{(\alpha k-d) D_{k,1}}{( (k-1)! )^2(C_k)^k} \int_{\R^d} \phi_1(x, \dots, x)\phi_2(x, \dots, x) |x|^{-\alpha (2k-1)} \d x \\[5pt]
&\hspace{-120pt} \text{ if } n f \big( R_k(C_kn/m)  \big) \to \infty, 
\end{cases}
\end{split}
\end{align}
where 
\begin{equation}  \label{e:def.Dkl}
D_{k, \ell} := \int_{(\R^d)^{2k - \ell - 1}} h_k(0, z_1, \dots, z_{k-1}) h_k(0, z_1, \dots, z_{\ell - 1}, z_k, \dots, z_{2k - \ell -1}) \d \bz, \ \ 1 \le \ell \le k. 
\end{equation}
Furthermore, denote $V_k(\phi):= V_k(\phi, \phi)$. 

\begin{remark}
As stated in  Lemma \ref{l:no.oscillation} in the Appendix, one can classify the limiting behavior of $n f \big( R_k(C_k n/m) \big)$ in terms of the slowly varying  component $L$ of $f$, together with the parameter $\beta$ in \eqref{e:RV.m(n)}. However,   since the resulting conditions on $L$ and $\beta$ are technically involved, we define \eqref{e:def.tau_kn} and \eqref{e:def.Vk} directly in terms of the limiting value of $n f \big( R_k(C_k n/m) \big)$.
\end{remark}

Subsequently, for any measurable sets $A, B \subset E_k$, we define
$$
V_k(A, B) := V_k(\1_{A}, \1_B), \ \ \text{ and } \ \ V_k(A):=V_k(A,A). 
$$
One can check that the bivariate map $V_k(\cdot, \cdot)$ becomes a covariance kernel of a Gaussian random measure on $E_k$ (more precisely, a random measure on $E_k$, which is Gaussian on every relatively compact subset of $E_k$); see \cite{Horowitz:1986} for more information. Denote by $(\Phi_k)_{k \ge 1}$ the sequence of independent and centered Gaussian random measures, such that the covariance kernel of $\Phi_k$ is given by $V_k$ for each $k \ge 1$. 
Now we state the asymptotic normality of the layered tail empirical measure.

\begin{proposition} \label{p:CLT.measure}
For every $K \ge 1$, we have the  multivariate finite-dimensional asymptotic normality: for each $\phi_k$ in the space $C_c(E_k)$ of continuous functions on $E_k$ with compact support, we have as $n\to\infty$, 
$$
\Big( \tau_{k,n}^{-1/2} \big( \nu_{k,n} (\phi_k) -\E[\nu_{k,n}(\phi_k)]\big)\Big)_{k = 1}^K \Rightarrow \lp \Phi_k(\phi_k) \rp_{k = 1}^K.  
$$
\end{proposition}

\begin{remark}
Proposition \ref{p:CLT.measure} establishes finite-dimensional weak convergence of the ``signed" random measure 
\begin{equation} \label{e:signed.centered}
\Big( \tau_{k,n}^{-1/2} \big( \nu_{k,n}  -\E[\nu_{k,n}]\big) \Big)_{k=1}^K, 
\end{equation}
where $\E[\nu_{k,n}]$ represents the intensity measure of $\nu_{k,n}$. However, this does not directly imply weak convergence of \eqref{e:signed.centered} in the space $M(E_k)$ of signed Radon measures on $E_k$. This is in stark contrast to the weak convergence in $M_+(E_k)$ (i.e., the space of non-negative Radon measures on $E_k$), which can be implied by finite-dimensional weak convergence. The fundamental reason for this discrepancy is that $M(E_k)$ is not generally metrizable (\cite[Section 1.3.1]{Barrio;Deheuvels;Geer:2007}), whereas $M_+(E_k)$ can be metrized under the vague topology (\cite[Proposition 3.17]{resnick:1987}). 
\end{remark}

Our next goal is to deduce  the asymptotic normality of the process \eqref{e:nu.kn.proc.version}. 

\begin{proposition} \label{p:CLT.functional}
For every $K \ge1$, it holds that
$$
\bigg(\Big(  \tau_{k,n}^{-1/2} \big(  \nu_{k,n}(t) - \E [ \nu_{k,n}(t) ] \big), \,  t \in (0, \infty]\Big) \bigg)_{k = 1}^K \Rightarrow \Big( \big( W_k(t), \, t \in (0, \infty] \big)\Big)_{k = 1}^K \text{ as } n \to \infty,
$$
in the space $\big(\mathcal{D} (0, \infty]\big)^K$, where, for each $1 \leq k \leq K$,
$$
\begin{aligned}
W_{k}(t) := \begin{cases}
B_{k,k} (L_{k,k} t^{d - \alpha k}) & \text{ if } n f \big( R_k(C_kn/m)  \big) \to 0, \\[3pt]
\sum_{\ell = 1}^k \xi^{(k - \ell)/2} B_{k, \ell}(L_{k,\ell} t^{d - \alpha(2k - \ell)} ) & \text{ if } n f \big( R_k(C_kn/m) \big) \to \xi \in (0, \infty), \\[3pt]
B_{k, 1} (L_{k, 1} t^{d - \alpha (2k - 1)} ) & \text{ if } n f \big( R_k(C_kn/m) \big) \to \infty.
\end{cases}
\end{aligned}
$$
Here, $(B_{k, \ell}, \, 1 \le \ell \le k \le K)$ are independent standard Brownian motions and  
$$
L_{k, \ell} := \binom{k}{\ell} \frac{\alpha k -d}{(k-\ell)! (\alpha (2k - \ell) - d)} \cdot \frac{D_{k,\ell}}{D_{k,k}},  \qquad  1 \leq \ell \leq k \le K.
$$
\end{proposition}
From Propositions \ref{p:CLT.measure} and \ref{p:CLT.functional}, we observe that the limiting Gaussian random measure and Gaussian process exhibit a phase transition based on the limit of $nf\big( \Rk \big)$.
Finally, Theorem \ref{t:CLT.Hill.estimator} establishes the asymptotic normality of the layered Hill estimator as desired.

\begin{theorem}
\label{t:CLT.Hill.estimator}
For every $K \ge1$, we have as $n\to\infty$, 
$$
\bigg( m^k \tau_{k,n}^{-1/2} \Big( H_{k,m,n} - m^{-k} \int_1^\infty \E \big[ \nu_{k,n}(s) \big] \Big|_{s=\frac{\consisest t}{R_k(C_kn/m)}} \frac{\d t}{t}  \Big)\bigg)_{k=1}^K \Rightarrow \Big( \int_1^\infty W_k(t) \frac{\d t}{t} \Big)_{k=1}^K. 
$$
\end{theorem}
\medskip

\subsection{Non-random centering}
\label{sec:non.random.centering}

Theorem \ref{t:CLT.Hill.estimator} has verified the asymptotic normality of the layered Hill estimators, but the result is not yet sufficiently practical for real applications. Specifically, the centering term in Theorem \ref{t:CLT.Hill.estimator} remains random, as it involves the order statistics $\consisest$. Moreover, the centering depends on $n$, which makes it difficult to construct confidence intervals for the tail exponent. To address this issue, we impose a set of conditions analogous to second-order regular variation, as described in \cite[Chapter 9]{Resnick:2007} and \cite{Geluk;deHaan;Resnick;Starica:1997}. It is worth noting that our conditions are applied directly to the  density function $f$, in contrast to  the cited works above, where regular variation is assumed on distribution functions. Specifically, for each $k \ge 1$, we assume the following three conditions.

\begin{itemize}

\item[(C1)] \label{cond: condition 1} Given $\delta_0>0$, we have, for every $\delta \geq \delta_0$,
$$
\begin{aligned}
& m^{k/2} \int_{ \rho = \delta}^{\infty} \int_{\theta \in \mathbb{S}^{d-1}} \int_{\bz \in (\R^d)^{k-1}} h_k(0, \bz)\,  \rho^{d-1} \\
& \times \left| \frac{f \big( R_k(C_kn/m) \rho  \big)}{f \big( R_k(C_kn/m)  \big)} \prod_{i=1}^{k-1} \frac{f \big( R_k(C_kn/m) | \rho \theta + R_k(C_kn/m)^{-1} z_i |  \big)}{f \big( R_k(C_kn/m)  \big)} - \rho^{-\alpha k } \right| \d \bz \d \sigma(\theta) \d \rho \\
& \to 0,
\end{aligned}
$$
as $n\to\infty$. Moreover, there exist $N = N(\delta_0) \ge1$,  $C = C(\delta_0) > 0$, and $q > 0$, such that for all $n \geq N$ and $\delta \geq \delta_0$, the left-hand side above is bounded by $C\delta^{-q}$. 
\item[(C2)] \label{cond: condition 2} It holds that
$$
m^{k/2} \lb \lp \frac{C_kn}{m} \rp^k R_k(C_kn/m)^d f \big( R_k(C_kn/m)  \big)^k - (\alpha k -d) \rb \to 0  \; \text{ as } n \to \infty.
$$
\item[(C3)] \label{cond: condition 3} It holds that
$$
\frac{m^{k/2}}{R_k(C_kn/m)} \to 0 \; \text{ as } n \to \infty.
$$
\end{itemize}

We observe that Condition (C1) can be regarded as a stronger version of the regular variation condition in \eqref{e:RV.f}, while  (C2) imposes a stronger requirement on the rate of convergence in \eqref{e:def.radius.function.R}. At first glance, this set of conditions may appear rather restrictive; however, it is straightforward to verify that a simple power-law density
$$
f(x) = C |x|^{-\alpha} \1 \{ |x| \ge 1 \}, \ \ \ x \in \R^d,
$$
satisfies all of these conditions.

\begin{theorem} \label{t:adjust.centering}
Let $K \ge1$. Under the conditions (C1)-(C3), as $n\to\infty$
\begin{equation}  \label{e:CLT.non-random.centering}
\lp m^k \tau_{k,n}^{-1/2} \lp H_{k,m,n} - \frac{1}{\alpha k -d} \rp \rp_{k=1}^K \Rightarrow \lp Z_k \rp_{k = 1}^K,  \ \  \text{ in } \R^K, 
\end{equation}
where $(Z_k)_{k=1}^K$ is a sequence of independent Gaussian random variables with zero mean,  and the variance of $Z_k$ depends on the limit of $nf(R_k(C_k n/m))$; more specifically, 
$$
\begin{aligned}
Z_k \sim  \begin{cases}
N (0, A_{k,k,\alpha}) & \text{ if } n f \big( R_k(C_kn/m)  \big) \to 0, \\[3pt]
N(0, \sum_{\ell = 1}^k \xi^{k-\ell} A_{k, \ell, \alpha}) & \text{ if } n f \big( R_k(C_kn/m)  \big) \to \xi \in (0, \infty), \\[3pt]
N(0, A_{k, 1, \alpha}) & \text{ if } n f \big( R_k(C_kn/m)  \big) \to \infty, 
\end{cases}
\end{aligned}
$$
where $A_{k, \ell, \alpha}$ is a constant defined by
$$
A_{k, \ell, \alpha} := L_{k, \ell} \cdot \frac{(\alpha (2k-\ell)-d)^2 -2\alpha (k-\ell)(\alpha k-d)}{(\alpha (2k - \ell) -d)^2 (\alpha k - d)^2}, \ \ 1 \le \ell \le k \le K. 
$$
\end{theorem}
For the actual calculation of the confidence interval for $\alpha$, it is necessary to replace the unknown parameter $\alpha$ in $\tau_{k,n}$ and the constant $A_{k,\ell,\alpha}$ with its consistent estimator $\hat{\alpha} = d/k + (k H_{k,m,n})^{-1}$, as suggested in \eqref{e:consis.alpha}. A detailed procedure is given in Section \ref{sec:coverage.rate}.
\medskip

\section{Simulation Studies}
\label{sec:simulation}

\subsection{Point estimates}

We present simulation results comparing the performance of the layered Hill estimators against the traditional Hill estimator. These results show that the layered Hill estimator significantly outperforms the traditional one, especially when some of the extremes are missing.
Tables \ref{table: pareto 2.5}, \ref{table: pareto 5}, and \ref{table: pareto 7.5} below present the results when data is generated from a power-law density in $\R^2$, defined by 
\begin{equation}  \label{e:power-law.simulation}
f(x) = C|x|^{-\alpha} \1 \{ |x|\ge1 \}, \ \ \ x\in \R^2. 
\end{equation}
The parameter values of $\alpha$ used in Tables \ref{table: pareto 2.5}, \ref{table: pareto 5}, and \ref{table: pareto 7.5} are 2.5, 5, and 7.5, respectively. For each value of $\alpha$, we set $n=10,000$ and consider three different choices for $m=m(n)$: $(i)$ $m=n^{0.1}$, $(ii)$ $m=n^{0.3}$, and $(iii)$ $m=n^{0.5}$. To assess the impact of missing extremes, we examine three cases with different rates $\delta$ of missing extremes. Specifically, we consider $(i)$ no missing data ($\delta=0$), $(ii)$ removing the largest $0.5m$ extreme points, measured by Euclidean distance from the origin, and calculating the layered Hill estimators from the remaining points ($\delta=0.5$), and $(iii)$ removing the largest $m$ extreme points and calculating the layered Hill estimators from the remaining points ($\delta=1$). 
We use the function $h_1(x)\equiv 1$, $x\in \R^2$ for the first layered Hill estimator, while the indicator $h_2(x_1, x_2) = \1 \big\{  |x_1-x_2|\le 1\big\}$,  $x_1, x_2 \in \R^2$, is used for the second layered Hill estimator. 

In Tables \ref{table: pareto 2.5}, \ref{table: pareto 5}, and \ref{table: pareto 7.5}, each row corresponds to the choice of $m=n^{0.1}$, $n^{0.3}$, and $n^{0.5}$, respectively. Each column presents the simulation results for each type of estimators and the proportion of missing extremes.  Here, “L1” stands for the first layered Hill estimator $H_{1,m,n}$, and “L2” refers to the second layered Hill estimator $H_{2,m,n}$. Additionally, “Mix”  stands for a linear combination defined  as $0.5 H_{1,m,n} + 0.5 H_{2,m,n}$. In each cell, the value without parentheses shows the average of the estimates over $500$ iterations under the same simulation setting,  while the value in parentheses represents their root mean squared error. 

Table \ref{table: pareto 2.5} illustrates that the traditional Hill estimator, equivalent to the first layered Hill estimator (hereafter we call it ``L1"), lacks robustness as the proportion $\delta$ of missing extremes increases. Indeed, the estimates increase as $\delta$ grows,  regardless of the choice of the cut-off sequence $m$. In contrast, the second layered Hill estimator (it is called ``L2" in the following) exhibits much greater stability, yielding estimates close to the true value of $\alpha$, even when $\delta=1$. This improved stability arises from the fact that L2 relies on edges within the second layer, making it less affected by missing extremes in the first layer; see Figure \ref{fig:layer.intuitive}. Essentially, the probability of missing extremes forming edges within the first layer is negligible; the first layer is located farther  from the origin compared to the second layer, so we do not observe, asymptotically,  any  pair of extremes forming an edge within the first layer. Tables \ref{table: pareto 5} and \ref{table: pareto 7.5} further confirm that L2  consistently outperforms L1. Across all scenarios, L2 maintains its stability as $\delta$ increases, producing accurate estimates of $\alpha$, whereas L1 becomes less reliable as $\delta$ grows. 

\begin{table}[t]
\centering
\begin{minipage}{\textwidth}
\centering
\caption{Power law, $\alpha = 2.5$}
    \label{table: pareto 2.5}
    \vspace{-2pt}
    \resizebox{0.85\textwidth}{!}{%
    \begin{tabular}{ c | c  c  c | c  c  c | c  c  c}
    \hline \hline
    & L1  & L1  & L1  & L2 & L2 & L2 & Mix & Mix & Mix \\
    & $\delta=0$ & $\delta=0.5$ & $\delta=1$ & $\delta=0$ & $\delta=0.5$ & $\delta=1$ & $\delta=0$ & $\delta=0.5$ & $\delta=1$ \\
    \hline
    $m=n^{0.1}$ & $2.926$  & $3.573 $ & $4.226$ & $2.565$ & $ 2.565$ & $2.565$ & $2.746$ & $3.069$  & $3.396$ \\
    & $(1.169)$ & $(2.387)$ & $(3.282)$  & $(0.241)$ & $(0.241)$ & $(0.241)$   & $(0.705)$ & $(1.314)$ & $(1.762)$\\
    \hline
    $m=n^{0.3}$ & $2.531$ & $3.173 $ & $3.665$  & $2.482$ & $2.482$ & $2.482$ & $2.507$ & $2.828$ & $3.074$\\
    & $(0.145)$ & $(0.754)$ & $(1.261)$  & $(0.070)$ & $(0.070)$ & $(0.070)$ & $(0.108)$  & $(0.412)$  & $(0.666)$\\
    \hline
    $m=n^{0.5}$ & $2.501$ & $3.109$ & $3.644$ & $2.428$ & $2.428$ & $2.428$  & $2.465$  & $2.769$ & $2.536$\\
    & $(0.055)$ & $(0.621)$ & $(1.162)$ & $(0.079)$ & $(0.079)$ & $(0.079)$  & $(0.067)$  & $(0.350)$ & $(0.621)$ \\
    \hline \hline 
    \end{tabular}
    }
\end{minipage}

\vspace{3em}  

\begin{minipage}{\textwidth}  
\centering
\caption{Power law, $\alpha = 5$}
    \label{table: pareto 5}
    \vspace{-2pt}
    \resizebox{0.85\textwidth}{!}{%
    \begin{tabular}{ c | c  c  c | c  c  c | c  c  c}
   \hline \hline
    & L1  & L1  & L1  & L2 & L2 & L2 & Mix & Mix & Mix\\
    & $\delta=0$ & $\delta=0.5$ & $\delta=1$ & $\delta=0$ & $\delta=0.5$ & $\delta=1$ & $\delta=0$ & $\delta=0.5$ & $\delta=1$ \\
    \hline
    $m=n^{0.1}$ & $7.530$  & $11.644$ & $17.145$  & $5.222$ & $  5.222$ & $5.223$ & $6.376$ & $8.433$ & $11.184$\\
    & $(12.682)$ & $(17.921)$ & $(30.020)$  & $(2.880)$ & $(2.880)$ & $(2.883)$  & $(7.781)$ & $(10.401)$ & $(16.452)$ \\
    \hline
    $m=n^{0.3}$ & $5.226$ & $9.045$ & $11.808 $   & $4.839$ & $4.843 $ & $4.852 $ & $5.033$  & $6.944$ & $8.330$\\
    & $(2.889)$ & $(6.842)$ & $(9.779)$  & $(2.355)$ & $(2.359)$ & $(2.369)$ & $(2.622)$  & $(4.601)$ & $(6.074)$\\
    \hline
    $m=n^{0.5}$ & $ 5.021$ & $8.701 $ & $ 11.848$ & $4.766$ & $ 4.771 $ & $ 4.813$ & $4.894$ & $6.736$ & $8.331$ \\
    & $(2.542)$ & $(6.249)$ & $(9.403)$  & $(2.271)$ & $( 2.278)$ & $(2.367)$  & $(2.407)$ & $(4.264)$ & $(5.885)$ \\
    \hline \hline
    \end{tabular}
    }
\end{minipage}

\vspace{3em}  

\begin{minipage}{\textwidth}  
\centering
\caption{Power law, $\alpha = 7.5$}
\label{table: pareto 7.5}
\vspace{-2pt}
\resizebox{0.85\textwidth}{!}{%
\begin{tabular}{ c | c  c  c | c  c  c | c  c  c}
\hline \hline
& L1  & L1  & L1  & L2 & L2 & L2 & Mix & Mix & Mix\\
& $\delta=0$ & $\delta=0.5$ & $\delta=1$ & $\delta=0$ & $\delta=0.5$ & $\delta=1$ & $\delta=0$ & $\delta=0.5$ & $\delta=1$ \\
\hline
$m=n^{0.1}$ & $12.510$  & $21.169$ & $26.509$  & $7.911$ & $7.974$ & $8.090  $ & $10.211$  & $14.572$ & $17.300$\\
& $(15.375)$ & $(31.065 )$ & $(35.924)$  & $(5.646)$ & $(5.703)$ & $(5.825)$  & $(10.511)$ & $(18.384)$ & $(20.875)$ \\
\hline
$m=n^{0.3}$ & $7.874$ & $14.792$ & $ 20.026$  & $7.271$ & $7.383$ & $ 7.635 $ & $7.573$ & $11.088$ & $13.831$\\
& $(5.670)$ & $(12.848 )$ & $(18.339)$ & $(4.802)$ & $(4.911)$ & $(5.170)$ & $(5.236)$ & $(8.880)$ & $(11.755)$ \\
\hline
$m=n^{0.5}$ & $ 7.521$ & $14.184$ & $20.003 $  & $7.141$ & $ 7.358$ & $7.794 $ & $7.331$  & $10.771$ & $13.899$\\
& $(5.055)$ & $(11.752)$ & $(17.601)$  & $( 4.644)$ & $(4.863 )$ & $(5.308)$  & $(4.850)$ & $(8.308)$ & $(11.459)$ \\
\hline \hline
\end{tabular}
}
\end{minipage}
\end{table}

\begin{figure}[t]
\centering
\includegraphics[scale=0.40]{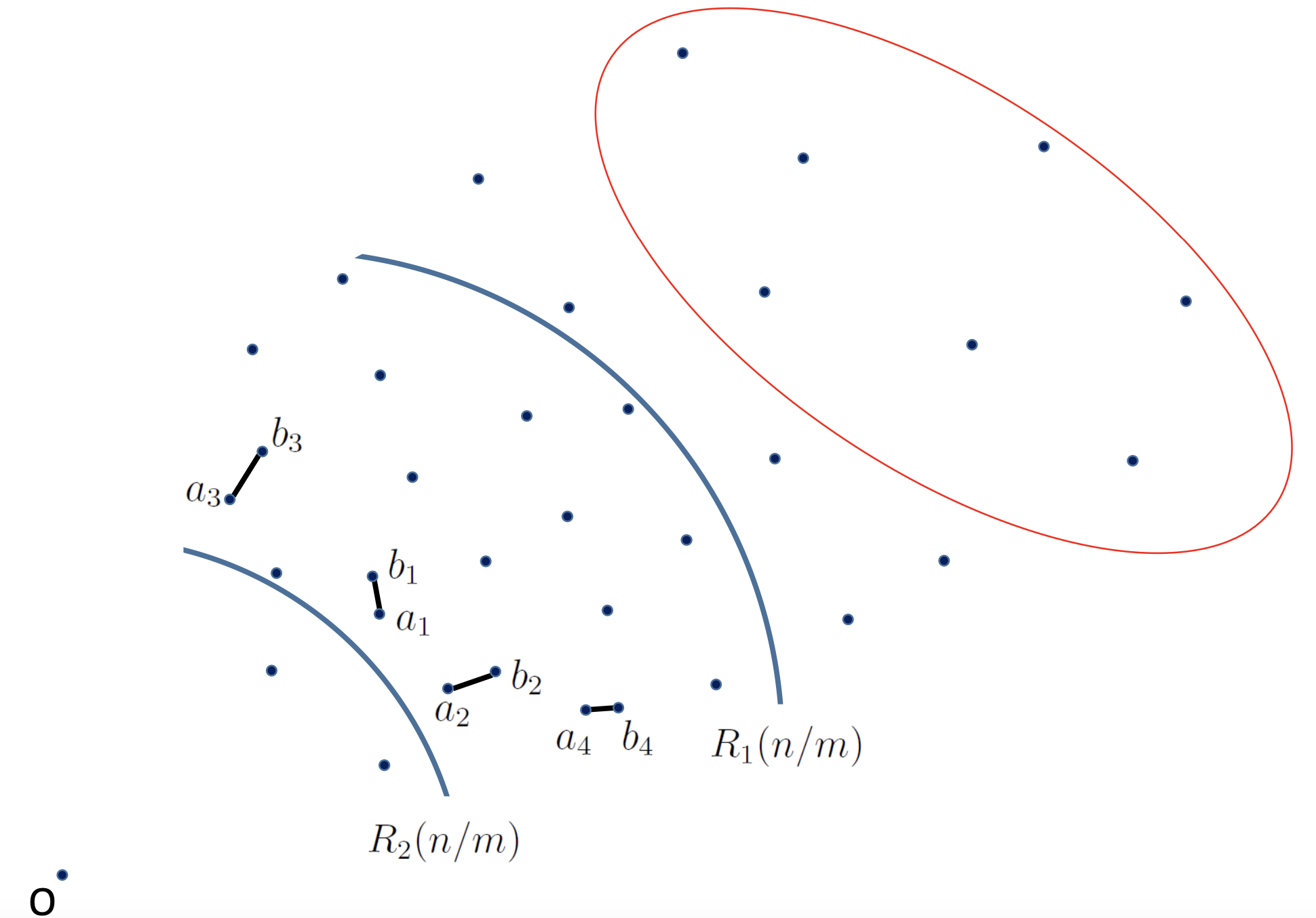}
\caption{\footnotesize{Extreme random points, circled in red in the first layer, are removed.  Then, the first layered Hill estimator  exhibits a significant bias, as it relies on these missing extremes. In contrast, the second layered Hill estimator only uses edges $\{ a_i, b_i \}$, $i=1,\dots,4$, and thus remains unaffected by the missing extreme points.}}
\label{fig:layer.intuitive}
\end{figure}
Additionally, the layered Hill estimator shows relatively robust performance even when the underlying distribution is not Pareto. In Table \ref{table:stable0.5}, a random sample is drawn from a spherically symmetric $\alpha$-stable law in $\R^2$ with $\alpha = 0.5$. In this case, L1 is unstable and increases as $\delta$ becomes larger. In contrast, L2 remains generally stable, estimating $\alpha$ more accurately, at least when $m = n^{0.1}$. Similarly, Table \ref{table:frechet0.5} presents simulation results where random points are sampled from the $\alpha$-Fr\'echet law with $\alpha = 0.5$. Here, the performance of L2 is comparable to that in Table \ref{table:stable0.5}. 
In general, the performance of L1 becomes less reliable when the underlying distribution is not Pareto (\cite[Section 4.4]{Resnick:2007}). This phenomenon appears to extend to L2 as well. Moreover, it is known that the performance of L1 is highly sensitive to the choice of the cut-off sequence $m=m(n)$ (see again  \cite[Section 4.4]{Resnick:2007}). An inappropriate choice of $m(n)$ may thus lead to substantial bias.  Tables \ref{table:stable0.5} and \ref{table:frechet0.5} suggest  that this issue arises not only for L1 but also for L2. 


\begin{table}[t]
\centering
\begin{minipage}{\textwidth}  
\centering
\caption{Stable, $\alpha = 0.5$}
\label{table:stable0.5}
\vspace{-2pt}
\resizebox{0.85\textwidth}{!}{%
\begin{tabular}{ l |c  c  c | c  c  c | c  c  c}
\hline \hline
& L1  & L1  & L1  & L2 & L2 & L2 & Mix & Mix & Mix\\
& $\delta=0$ & $\delta=0.5$ & $\delta=1$ & $\delta=0$ & $\delta=0.5$ & $\delta=1$ & $\delta=0$ & $\delta=0.5$ & $\delta=1$ \\
\hline
$m=n^{0.1}$ & $1.097$  & $1.763$ & $2.366$  & $0.511$ & $ 0.511$ & $0.511$ & $0.804$  & $1.137$ & $1.439$\\
& $(3.350)$ & $(3.363)$ & $(4.262)$   & $(0.237)$ & $(0.237 )$ & $(0.237)$ & $(1.794)$ & $(1.800)$ & $(2.250)$  \\
\hline
$m=n^{0.3}$ & $0.542$ & $1.167$ & $1.639$ & $0.360$ & $0.360$ & $0.360$ & $0.451$ & $0.764$ & $1.000$ \\
& $(0.162)$ & $(0.735)$ & $(1.239)$  & $(0.155)$ & $(0.155)$ & $(0.155)$ & $(0.159)$ & $(0.445)$ & $(0.697)$\\
\hline
$m=n^{0.5}$ & $ 0.503$ & $1.114$ & $1.611$  & $0.174$ & $0.174$ & $0.174$ & $0.339$ & $0.644$ & $0.893$ \\
& $(0.052)$ & $(0.623)$ & $(1.121)$   & $( 0.328)$ & $(0.328)$ & $(0.328)$ & $(0.190)$ & $(0.471)$ & $(0.725)$ \\
\hline \hline
\end{tabular}
}
\end{minipage}

\vspace{3em}  

\begin{minipage}{\textwidth}  
\centering
\caption{Fr\'echet, $\alpha = 0.5$}
\label{table:frechet0.5}
\vspace{-2pt}
\resizebox{0.85\textwidth}{!}{%
\begin{tabular}{ l | c  c  c | c  c  c | c  c  c}
\hline \hline
& L1  & L1  & L1  & L2 & L2 & L2 & Mix & Mix & Mix\\
& $\delta=0$ & $\delta=0.5$ & $\delta=1$ & $\delta=0$ & $\delta=0.5$ & $\delta=1$ & $\delta=0$ & $\delta=0.5$ & $\delta=1$ \\
\hline
$m=n^{0.1}$ & $0.965$  & $1.684$ & $2.385$  & $0.491$ & $0.491$ & $0.491$ & $0.593$  & $1.088$ & $1.438$ \\
& $(1.365)$ & $(2.832)$ & $(3.979)$ & $(0.219)$ & $(0.219)$ & $(0.219)$   & $(0.792)$  & $(1.526)$  & $(2.099)$ \\
\hline
$m=n^{0.3}$ & $0.539 $ & $1.174$ & $1.683$ & $0.348$ & $0.348$ & $0.348$ & $0.444$ & $0.761$ & $1.016$ \\
& $(0.143)$ & $(0.762)$ & $(1.294)$  & $(0.167)$ & $(0.167)$ & $(0.167)$ & $(0.155)$ & $(0.465)$ & $(0.731)$\\
\hline
$m=n^{0.5}$ & $ 0.503$ & $1.107$ & $1.630$   & $0.154$ & $0.154$ & $0.154$ & $0.329$ & $0.631$ & $0.892$\\
& $(0.051)$ & $(0.614)$ & $(1.140)$  & $( 0.341)$ & $(0.341)$ & $(0.341)$ & $(0.196)$ & $(0.478)$ & $(0.741)$ \\
\hline \hline
\end{tabular}
}
\end{minipage}
\end{table}

\subsection{Asymptotic normal curve}

Figures \ref{fig:normal.density.delta=0}, \ref{fig:normal.density.delta=0.5}, and \ref{fig:normal.density.delta=1} below investigate how well the empirical densities constructed from Theorem \ref{t:adjust.centering} fit the standard normal curve. We set $n=10,000$, $m=n^{0.3}$, and use the density function  \eqref{e:power-law.simulation} with $\alpha=2.5$, considering three different missing rates: $\delta = 0, 0.5$, and $1$. The red, blue, and purple curves represent the empirical densities of the first layered (L1), second layered (L2), and mixture layered (Mix) Hill estimators, respectively. 
In each case, Theorem \ref{t:adjust.centering} suggests how to normalize our estimators, but some of the constants in \eqref{e:CLT.non-random.centering} still depend on the unknown parameter $\alpha$. Therefore, for normalization, we need to replace $\alpha$ in these constants with $\hat{\alpha}$, a consistent estimator  obtained from the left-hand side in \eqref{e:consis.alpha}. 

Additionally, we need to determine which variance term in Theorem \ref{t:adjust.centering} should be used for normalization—this depends on whether the limit of $nf\big( \Rk \big)$ is $0$, constant, or infinite. 
For this purpose, a heuristic yet practical algorithm is proposed. Since the density $f$ has a regular variation exponent $\alpha$, \eqref{e:def.radius.function.R} implies, heuristically, that by ignoring possible slowly varying terms in \eqref{e:RV.f}, one can see that $n^kR_k(n)^{d-\alpha k}$ is asymptotically constant as $n\to\infty$. That is, $R_k(n) \approx n^{k/(\alpha k-d)}$ up to constant factors when $n$ is large enough. Given this regular variation assumption and the heuristic asymptotics, if we set $m=n^\beta$ for some $\beta\in (0,1)$, then for some constant $C>0$ and large enough $n$, 
$$
nf\big( \Rk \big) \approx C n\Rk^{-\alpha} \approx Cn^{(\alpha\beta k -d)/(\alpha k-d)}.
$$ 
If $\hat{\alpha}$ denotes a (consistent) estimate of $\alpha$ obtained from \eqref{e:consis.alpha},  we can expect that
\begin{align*}
n f \lp \Rk \rp \to \begin{cases}
0 & \text{ if } \beta < d/(\hat \alpha k), \\
\text{constant} & \text{ if } \beta = d/(\hat \alpha k), \\
\infty & \text{ if } \beta > d/(\hat \alpha k).
\end{cases}
\end{align*}
Once the regime in Theorem \ref{t:adjust.centering} is determined, we replace $\alpha$ in $\tau_{k,n}$ and the constant $A_{k,\ell,\alpha}$ with its estimate $\hat{\alpha}$ and compute kernel density curves for the normalized layered Hill estimators.

In the current simulation scheme, we have $k = 1$ or $2$, $d = 2$, and $\beta = 0.3$. Substituting these values, along with the estimate $\hat{\alpha}$, which is approximately $2.5$ as shown in Table \ref{table: pareto 2.5}, we conclude that the first regime in Theorem \ref{t:adjust.centering} (i.e., $nf\big( \Rk \big) \to 0$) applies to both $k = 1$ and $k = 2$. Figure \ref{fig:normal.density.delta=0} illustrates the case with no missing extremes, i.e., $\delta = 0$. In this setting, the kernel density curves for both L1 and L2 shows a good fit to the standard normal density. However, as seen in Figures \ref{fig:normal.density.delta=0.5} and \ref{fig:normal.density.delta=1}, the kernel density curve for L1 becomes significantly biased when extremes are missing, while L2 continues to appropriately capture the standard normal density.

\begin{figure}[t]
\begin{minipage}{\textwidth}  
\centering
\includegraphics[scale=0.452185]{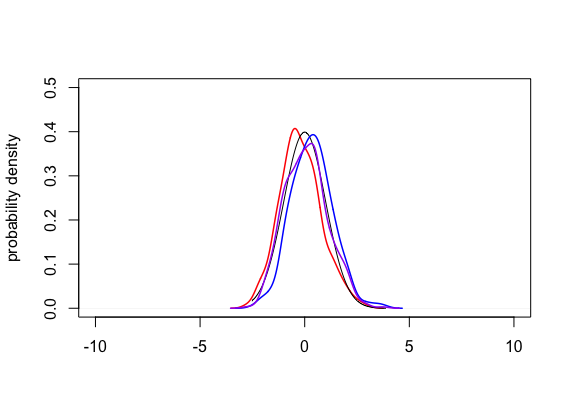}
\vspace{-30pt}
\caption{\footnotesize{Kernel density curves of the normalized layered Hill estimators without missing values (i.e., $\delta = 0$). The black curve represents the  density function of the standard normal distribution. The red curve is the kernel density estimate for the first layered Hill estimator, the blue curve for the second layered Hill estimator, and the purple curve for the mixture of the two.}}
\label{fig:normal.density.delta=0}
\end{minipage}

\vspace{3.5pt}

\begin{minipage}{\textwidth}  
\centering
\includegraphics[scale=0.452185]{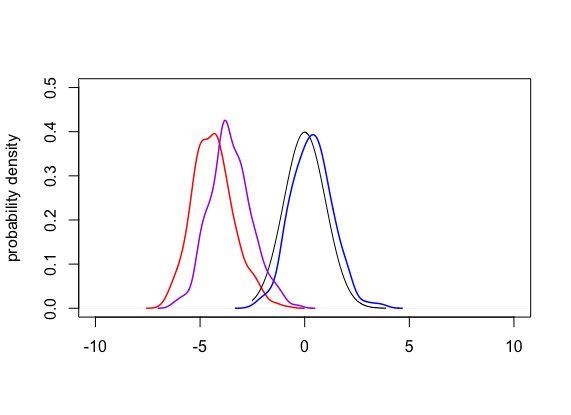}
\vspace{-30pt}
\caption{\footnotesize{ Kernel density curves of the normalized layered Hill estimators with  missing rate $\delta = 0.5$.}}
\label{fig:normal.density.delta=0.5}
\end{minipage}

\vspace{3.5pt}

\begin{minipage}{\textwidth}  
\centering
\includegraphics[scale=0.452185]{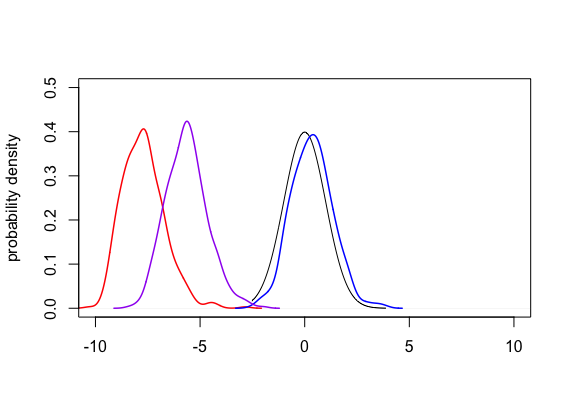}
\vspace{-30pt}
\caption{\footnotesize{Kernel density curves of the normalized layered Hill estimators with  missing rate $\delta = 1$.}}
\label{fig:normal.density.delta=1}
\end{minipage}
\end{figure}

\subsection{Coverage rates for confidence interval}  \label{sec:coverage.rate}

Next we propose a method to construct an asymptotic confidence interval for the tail exponent $\alpha$. In the same simulation setting as  the last subsection, the discussion below assumes the first regime (i.e., $nf\big(  \Rk\big) \to 0$) in Theorem \ref{t:adjust.centering}. Then, $\tau_{k,n}=m^k$, and let $A_{k,k,\hat{\alpha}}$ be the estimated value of $A_{k,k,\alpha}$ obtained by replacing $\alpha$ with $\hat{\alpha}$. 
Given a standard normal random variable $Z$ and a confidence level $\gamma \in (0,1)$, define $c_L := -z_{1 - \gamma/2}$ and $c_U := z_{1 - \gamma/2}$, where $z_{1- \gamma/2}$ is the $(1 - \gamma/2)$-quantile of $Z$. Then, $\Pr (c_L \le Z \le c_U) = \gamma$ and  Theorem \ref{t:adjust.centering}  implies that 
$$
\Pr \Big(c_L \le m^{k/2} A_{k,k,\hat{\alpha}}^{-1/2} \big( H_{k,m,n}-(\alpha k - d)^{-1} \big) \le c_U \Big)
$$
is asymptotically equal to $\gamma$. Consequently, the corresponding asymptotic confidence interval for $\alpha$ is given by 
\begin{align*}
\frac{1}{k} \left( \frac{1}{H_{k,m,n}-c_L m^{-k/2}A_{k,k,\hat{\alpha}}^{1/2}} + d \right) \le \alpha \le \frac{1}{k} \left( \frac{1}{H_{k,m,n}-c_U m^{-k/2}A_{k,k,\hat{\alpha}}^{1/2}} + d \right).
\end{align*}
\vspace{0pt}

Table \ref{table:coverage.rate} presents the coverage rates of the $95\%$ confidence interval, when random points are drawn from the Pareto distribution in \eqref{e:power-law.simulation} with $\alpha = 2.5$, $5$, and $7.5$. The setups for this table are identical to those in Tables \ref{table: pareto 2.5}, \ref{table: pareto 5}, and \ref{table: pareto 7.5}, respectively. According to Table \ref{table:coverage.rate}, when no extremes are missing, i.e., $\delta = 0$, the simulated coverage rates are close to $0.95$ for both estimators. However, as $\delta$ becomes positive, L1 fails to produce an accurate confidence interval due to substantial bias, as shown in Figures \ref{fig:normal.density.delta=0.5} and \ref{fig:normal.density.delta=1}. Despite this, L2 maintains high coverage rates close to $0.95$ even in the presence of missing extremes.

\begin{table}[t]
\centering
\caption{Coverage rate in Pareto case}
\label{table:coverage.rate}
\vspace{-2pt}
\begin{tabular}{ l|  c  c  c | c  c  c }
\hline \hline
& L1  & L1  & L1 &  L2 & L2 & L2  \\
& $\delta=0$ & $\delta=0.5$ & $\delta=1$ & $\delta=0$ & $\delta=0.5$ & $\delta=1$  \\
\hline
$\alpha = 2.5$ & $0.938$ & $0.016$ & $0.000$  & $0.930$ & $0.930$ & $0.930$   \\  
$\alpha = 5$ & $0.920$ & $0.022$ & $0.006$  & $0.914$ & $0.917$ & $0.915$ \\
$\alpha = 7.5$ & $0.911$ & $0.026$ & $0.008$  & $0.892$ & $0.896$ & $0.894$ \\
\hline \hline
\end{tabular}
\medskip
\end{table}

\medskip


\section{Proof of the Results}
\label{sec:proofs}

Throughout this section, denote by $C^*$ a generic and positive constant, which is independent of $n$ but may vary between and within the lines. We begin with the lemma regarding the expectation and covariance asymptotics of $\nu_{k,n}(\phi)$ for $\phi \in C_c(E_k)$ (i.e., the space of continuous functions on $E_k$ with compact support). 

\begin{lemma}  \label{l:exp.var.asym}
$(i)$ For every $k\ge1$ and $\phi \in C_c(E_k)$, 
$$
m^{-k} \E \big[ \nu_{k,n}(\phi) \big] \to \mu_k (\phi), \ \ \ n\to\infty. 
$$
$(ii)$ For every $k\ge1$ and $\phi_1, \phi_2 \in C_c(E_k)$, 
$$
\tau_{k,n}^{-1} \Cov \big( \nu_{k,n}(\phi_1), \, \nu_{k,n}(\phi_2) \big) \to V_k(\phi_1, \phi_2),  \ \ \ n\to\infty, 
$$
where $(\tau_{k,n})_{n\ge1}$ and  $V_k(\phi_1, \phi_2)$ are defined respectively in \eqref{e:def.tau_kn} and \eqref{e:def.Vk}. 
\end{lemma}

Given $\phi\in C_c(E_k)$, let 
\begin{equation}  \label{e:symmetrization}
\widetilde \phi(y_1,\dots,y_k) := \frac{1}{k!}\, \sum_{\sigma} \phi (y_{\sigma(1)}, \dots, y_{\sigma(k)}), \ \ \ (y_1,\dots,y_k) \in E_k, 
\end{equation}
be the symmetrization of $\phi$, where $\sigma$ ranges over all permutations on $\{ 1,\dots,k \}$. Since $\widetilde \phi$ is permutation invariant, one can define $\widetilde\phi(\Y):=\widetilde\phi(y_1,\dots,y_k)$ for every $k$-point set $\Y=\{y_1,\dots, y_k\}\subset \R^d$. Due to permutation invariance of $h_k$, it now holds that for any  $\phi\in C_c(E_k)$, 
\begin{equation}  \label{e:nukn.no.permutation.expression}
\nu_{k,n}(\phi) = \sum_{\Y\subset \Pn, \, |\Y|=k} h_k (\Y)\widetilde \phi\big( R_k(C_kn/m)^{-1}\Y \big). 
\end{equation}


\begin{proof}
\noindent{\underline{Proof of $(i)$}} \\
By the multivariate Mecke formula for Poisson point processes (see, e.g., \cite[Theorem 4.4]{Last;Penrose:2017}) and the expression in \eqref{e:nukn.no.permutation.expression}, 
\begin{equation}  \label{e:multi.Mecke}
m^{-k} \E \big[ \nu_{k,n}(\phi) \big] = \frac{1}{k!} \Big(\frac{n}{m} \Big)^k  \int_{(\R^d)^k} \widetilde \phi \lp R_k(C_k n/m)^{-1} \by \rp h_k(\by) \prod_{i=1}^k f(y_i) \d \by, 
\end{equation}
where $\by=(y_1,\dots,y_k)$,  $y_i\in \R^d$, $i=1\dots,k$. 
Performing change of variables by $x=y_1$ and $z_i = y_{i+1}-y_1$ for $i=1,\dots,k-1$ and using the translation invariance of $h_k$, the expression \eqref{e:multi.Mecke} can be written as 
\begin{align}
\frac{1}{k!} \Big(\frac{n}{m} \Big)^k  \int_{x \in \R^d} & \int_{\bz  \in (\R^d)^{k-1}}  \widetilde \phi \lp R_k(C_kn/m)^{-1} (x, x+z_1, \dots, x+ z_{k-1} ) \rp h_k(0, \bz) \, \label{e:first.change.of.variables} \\
& \times f(x) \prod_{i=1}^{k-1} f(x+z_i) \d \bz \d x.\notag
\end{align}
Applying the polar coordinate transform    $x = r \theta$ with $r \in (0, \infty)$ and $\theta \in \mathbb{S}^{d-1}$,  the above expression equals 
\begin{align}
\frac{1}{k!} \Big(\frac{n}{m}\Big)^k \int_{r =0}^{\infty} \int_{\theta \in \mathbb{S}^{d-1}} \int_{\bz \in (\R^d)^{k-1}} &  \widetilde \phi \lp R_k(C_kn/m)^{-1}(r \theta,  r \theta+z_1, \dots,  r \theta+z_{k-1} ) \rp \label{e:polar.coordinate.r} \\
& \times h_k(0, \bz) f(r ) \prod_{i=1}^{k-1} f(  r \theta+z_i) r^{d-1} \d \bz \d \sigma(\theta) \d r, \notag
\end{align}
where $\sigma$ denotes the surface measure on $\mathbb{S}^{d-1}$ induced from the Lebesgue measure on $\R^d$. By an additional change of variable $\rho=R_k(C_kn/m)^{-1}r$, \eqref{e:polar.coordinate.r} is further equal to 
\begin{align}
&\frac{1}{(C_k)^k k!} \Big(\frac{C_kn}{m}\Big)^k R_k (C_kn/m)^d f\big( R_k(C_kn/m) \big)^k \label{e:final.form.expectation}  \\
& \times \int_{\rho = 0}^{\infty} \int_{\theta \in \mathbb{S}^{d-1}} \int_{\bz \in (\R^d)^{k-1}}  \widetilde \phi \lp \rho \theta, \rho\theta + \Rk^{-1} z_1 , \dots, \rho \theta+ \Rk^{-1} z_{k-1}  \rp   \notag  \\
& \times h_k(0, \bz) \frac{f \lp R_k(C_kn/m) \rho  \rp}{f \lp R_k(C_kn/m)  \rp} \prod_{i=1}^{k-1} \frac{f \lp R_k(C_kn/m) |\rho\theta +  R_k(C_kn/m)^{-1}z_i |\rp}{f \lp R_k(C_kn/m) \rp} \rho^{d-1} \d \bz \d \sigma(\theta) \d \rho. \notag 
\end{align}
It follows from \eqref{e:RV.f} and \eqref{e:def.radius.function.R} that  the integrand in \eqref{e:final.form.expectation} converges to 
$$
\frac{\alpha k-d}{(C_k)^k k!} \, \phi(\rho \theta, \dots, \rho \theta) h_k(0,\bz) \rho^{-\alpha k +d-1}, \ \ \text{as } n\to\infty, 
$$
for every $\rho>0$, $\theta \in \mathbb{S}^{d-1}$, and $\bz\in (\R^d)^{k-1}$. 
Since the function $\phi$ has compact support, the range of $\rho$ can be restricted to the interval $[\delta, \infty)$ for some $\delta > 0$. By this property and Potter's bound (see Lemma \ref{l:potter} in the Appendix), we obtain that for any $0 < \vep < \alpha - d$, 
\begin{equation}  \label{e:Potter1}
\frac{f \lp R_k(C_kn/m) \rho \rp}{f \lp R_k(C_kn/m) \rp}  \leq ( 1+ \vep) \rho^{-\alpha + \vep}, 
\end{equation}
and for each $i=1,\dots,k-1$, 
\begin{equation}  \label{e:Potter2}
\frac{f \lp R_k(C_kn/m) | \rho \theta + R_k(C_kn/m)^{-1} z_i|  \rp}{f \lp R_k(C_kn/m)^{-1}  \rp}  \leq C^*, 
\end{equation}
for all large enough $n$. Using the bounds in \eqref{e:Potter1} and \eqref{e:Potter2}, we can apply the dominated convergence theorem to conclude that  \eqref{e:final.form.expectation} converges to 
$$
\frac{\alpha k-d}{(C_k)^k k!}\, \int_{\rho = 0}^{\infty} \int_{\theta \in \mathbb{S}^{d-1}} \int_{\bz \in (\R^d)^{k-1}} \phi( \rho \theta, \dots, \rho \theta) h_k(0, \bz) \rho^{-\alpha k + d - 1} \d \bz \d \sigma(\theta) \d \rho = \mu_k (\phi), 
$$
as desired. 
\vspace{5pt}

\noindent{\underline{Proof of $(ii)$}} 
For a $k$-point set $\Y= \{ y_1,\dots,y_k \}\subset \R^d$, define 
$$
g_k^{(i)}(\Y) :=   h_k ( \mathcal{Y})\, \widetilde  \phi_i \big( R_k(C_kn/m)^{-1} \mathcal{Y} \big), \ \ \ i\in \{1,2\}, 
$$
where $\widetilde \phi_i$ is the symmetrization of $\phi_i$. Then, 
$$
\nu_{k,n}(\phi_i) = \sum_{\Y\subset \Pn, \, |\Y|=k} g_k^{(i)}(\Y), \ \ \ i\in \{1,2\}. 
$$
By the Mecke formula for Poisson point processes (see Lemma 8.1 in \cite{Owada:2017}), 
\begin{align}
&\tau_{k,n}^{-1}\Cov\big( \nu_{k,n}(\phi_1), \, \nu_{k,n}(\phi_2) \big) \label{e:Mecke.variance} \\
&= \tau_{k,n}^{-1} \, \sum_{\ell=0}^k \E \bigg[ \sum_{\Y\subset \Pn, \, |\Y|=k} \sum_{\substack{\Y'\subset \Pn, \, |\Y'|=k, \\ |\Y\cap \Y'|=\ell}} g_k^{(1)}(\Y)\, g_k^{(2)} (\Y')  \bigg] -\tau_{k,n}^{-1} \E\big[ \nu_{k,n}(\phi_1) \big]\E\big[ \nu_{k,n}(\phi_2) \big] \notag  \\
&= \tau_{k,n}^{-1} \, \sum_{\ell=1}^k \E \bigg[ \sum_{\Y\subset \Pn, \, |\Y|=k} \sum_{\substack{\Y'\subset \Pn, \, |\Y'|=k, \\ |\Y\cap \Y'|=\ell}} g_k^{(1)}(\Y)\, g_k^{(2)} (\Y')  \bigg] \notag \\
&= \sum_{\ell=1}^k  \frac{\tau_{k,n}^{-1} n^{2k-\ell}}{\ell! ((k-\ell)!)^2}\, \E \big[ g_k^{(1)} \big( \big\{ X_1,\dots,X_k\}\big) \, g_k^{(2)}\big( \{X_1,\dots, X_\ell, X_{k+1}, \dots, X_{2k-\ell}\}\big) \big], \notag 
\end{align}
where $X_1,\dots,X_{2k-\ell}$ are i.i.d.~random variables with common density $f$. The expectation term in \eqref{e:Mecke.variance} can be expressed as 
\begin{align*}
&\E \big[ g_k^{(1)} \big( \big\{ X_1,\dots,X_k\}\big) \, g_k^{(2)}\big( \{X_1,\dots, X_\ell, X_{k+1}, \dots, X_{2k-\ell}\}\big) \big] \\ 
&\quad = \int_{(\R^d)^{2k - \ell}} \widetilde \phi_1 \big( R_k(C_kn/m)^{-1}(y_1, \dots, y_k) \big)\,  h_k(y_1, \dots, y_k) \\
&\quad  \qquad \times \widetilde \phi_2 \big( R_k(C_kn/m)^{-1} ( y_1, \dots, y_{\ell}, y_{k+1}, \dots, y_{2k - \ell}) \big) \\
&\quad  \qquad \times h_k (y_1. \dots, y_{\ell}, y_{k+1}, \dots, y_{2k - \ell})  \prod_{i=1}^{2k - \ell} f(y_i) \d \by. 
\end{align*}
Now, we compute the limit of
\begin{align}
&\tau_{k,n}^{-1}n^{2k-\ell} \int_{(\R^d)^{2k-\ell}}\widetilde  \phi_1 \big( R_k(C_kn/m)^{-1}(y_1, \dots, y_k)\big)\,  h_k(y_1, \dots, y_k) \label{e:permutation.to.identity}\\
& \times \widetilde \phi_2 \big( R_k(C_kn/m)^{-1} ( y_1, \dots, y_\ell, y_{k+1}, \dots, y_{2k-\ell}) \big) \notag\\
& \times h_k (y_1. \dots, y_\ell, y_{k+1}, \dots, y_{2k - \ell})  \prod_{i=1}^{2k - \ell} f(y_i) \d \by, \notag
\end{align}
for every $\ell=1,\dots,k$. 
By repeating calculations based on  the same change of variables as in \eqref{e:first.change.of.variables}--\eqref{e:final.form.expectation}, the expression in \eqref{e:permutation.to.identity}  is equal to 
\begin{align*}
&\tau_{k,n}^{-1} n^{2k-\ell}R_k(C_kn/m)^d f \lp R_k(C_kn/m)  \rp^{2k-\ell} \int_{\rho = 0}^{\infty} \int_{\theta \in \mathbb{S}^{d-1}} \int_{\bz \in (\R^d)^{2k - \ell - 1}} \\
&\qquad \widetilde \phi_1 \big( \rho \theta, \rho \theta + R_k(C_kn/m)^{-1} z_1, \dots, \rho \theta + R_k(C_kn/m)^{-1} z_{k-1} \big)\,  h_k(0, z_1, \dots, z_{k-1}) \\
&\qquad \times \widetilde \phi_2 \big( \rho \theta, \rho \theta  + R_k(C_kn/m)^{-1} z_1, \dots, \rho \theta  + R_k(C_kn/m)^{-1} z_{\ell - 1}, \\
&\qquad  \qquad\qquad \qquad\qquad \qquad\rho \theta  + R_k(C_kn/m)^{-1} z_k, \dots, \rho \theta +  R_k(C_kn/m)^{-1} z_{2k - \ell - 1}  \big) \\
&\qquad  \times h_k(0, z_1, \dots, z_{\ell - 1}, z_k, \dots, z_{2k - \ell - 1}) \frac{f \lp R_k(C_kn/m) \rho\rp}{f \lp R_k(C_kn/m)  \rp} \\
&\qquad \times  \prod_{i=1}^{2k - \ell - 1} \frac{f \lp R_k(C_kn/m) |\rho \theta + R_k(C_kn/m)^{-1}z_i|  \rp}{f \lp R_k(C_kn/m) \rp} \rho^{d-1} \d \bz \d \sigma(\theta) \d \rho.
\end{align*}
By the regular variation property \eqref{e:RV.f}, and condition \eqref{e:def.radius.function.R}, as well as the Potter bounds as in \eqref{e:Potter1} and \eqref{e:Potter2} (for the application of the dominated convergence theorem), we find that for every $\ell=1,\dots,k$, the expression above is asymptotically equal to, as $n \to \infty$,
\begin{align*}
&\frac{n^{2k-\ell}}{\tau_{k,n}} \Rk^d f\big(\Rk \big)^{2k-\ell} D_{k,\ell} \int_{\R^d} \phi_1(x,\dots,x) \phi_2(x,\dots,x)|x|^{-\alpha (2k-\ell)} \d x   \\
&\sim  \frac{m^k}{\tau_{k,n}} \Big( nf \big(\Rk \big)\Big)^{k-\ell} \frac{(\alpha k-d)D_{k,\ell}}{(C_k)^k}\, \int_{\R^d} \phi_1(x,\dots,x) \phi_2(x,\dots,x) |x|^{-\alpha (2k-\ell)} \d x, \ \ \  
\end{align*} 
where $D_{k,\ell}$ is a constant given in \eqref{e:def.Dkl}. We obtain, as $n \to \infty$, 
\begin{align} 
\begin{split} \label{e:limit.var.expression}
\tau_{k,n}^{-1} \Cov\big( \nu_{k,n}(\phi_1), \, \nu_{k,n}(\phi_2) \big) &\sim \sum_{\ell=1}^k \frac{\tau_{k,n}^{-1} m^k \big( nf(\Rk) \big)^{k-\ell}}{\ell! ((k-\ell)!)^2}\\
& \;  \times \frac{(\alpha k-d)D_{k,\ell}}{(C_k)^k} \int_{\R^d} \phi_1(x,\dots,x) \phi_2(x,\dots,x) |x|^{-\alpha (2k-\ell)} \d x. 
\end{split}
\end{align}
Note that which term on the right-hand side of \eqref{e:limit.var.expression} dominates depends on the behavior of $nf(\Rk)$ as $n \to \infty$. Specifically, if $nf(\Rk) \to 0$, the $k$th term (with $\ell = k$) dominates, while if $nf(\Rk) \to \infty$, the term with $\ell = 1$ becomes dominant. Furthermore, if $nf(\Rk) \to \xi \in (0, \infty)$, all terms will contribute in the limit. In all three cases, by the definition of $(\tau_{k,n})_{n \ge 1}$, it can be shown that \eqref{e:limit.var.expression} converges to $V_k(\phi_1, \phi_2)$.
\end{proof}

\subsection{Proof of Consistency}

Here we prove the results in Section \ref{sec:consistency}. 

\begin{proof}[Proof of Proposition \ref{p:consistency.tail.measure} $(i)$]
By Lemma \ref{l:exp.var.asym} $(i)$, we have $m^{-k} \E \big[ \nu_{k,n}(\phi) \big] \to \mu_k (\phi)$ as $n\to\infty$, for every $\phi\in C_c^+(E_k)$. It is easy to see that $\tau_{k,n}/m^{2k} \to 0$ as $n\to\infty$, and hence, Lemma \ref{l:exp.var.asym} $(ii)$ yields that 
$m^{-2k} \text{Var}\big(  \nu_{k,n}(\phi) \big) \to 0$ as $n \to \infty$. 
Now, the Chebyshev's inequality concludes that $m^{-k}\nu_{k,n}(\phi) \stackrel{\Pr}{\to} \mu_k(\phi)$, $n\to\infty$ for every $\phi\in C_c^+(E_k)$. 
\end{proof}
\vspace{5pt}

\begin{proof}[Proof of Proposition \ref{p:consistency.tail.measure} $(ii)$]
For the proof of \eqref{e:consistent.est.by.order.stat}, one can see that for every $\vep>0$, 
\begin{align} \label{e:conv.prob.order}
& \Pr \Big(\,  \Big|  \frac{\consisest}{R_k(C_kn/m)}-1 \Big| \ge \vep \Big) \\
& = \Pr \big( \consisest \ge (1+\vep)R_k(C_kn/m) \big) + \Pr \big( \consisest  \le (1-\vep)R_k(C_kn/m) \big) \notag \\
&= \Pr \big( \nu_{k,n}(B_{1+\vep}) \ge m^k \big) +\Pr \big( \nu_{k,n}(B_{1-\vep}) \le m^k \big),   \notag
\end{align}
where $B_{1\pm \vep}$ is defined at \eqref{e:def.Br}. 
By Proposition \ref{p:consistency.tail.measure} $(i)$, we have 
$$
m^{-k} \nu_{k,n}(B_{1\pm \vep}) \stackrel{\Pr}{\to} \mu_k (B_{1\pm \vep}) = (1\pm \vep)^{d-\alpha k}, \ \ \ n\to\infty, 
$$
and, hence, \eqref{e:conv.prob.order} vanishes as $n\to\infty$. 

For the proof of the second assertion, it follows from  Proposition \ref{p:consistency.tail.measure} $(i)$ and \eqref{e:consistent.est.by.order.stat} that 
$$
\Big( m^{-k}\nu_{k,n}, \frac{\consisest}{R_k(C_kn/m)} \Big) \stackrel{\Pr}{\to} (\mu_k, 1), \ \ \text{in } M_+(E_k) \times (0,\infty). 
$$
Since the map $G:M_+(E_k)\times (0,\infty) \to M_+(E_k)$, defined by $G(\mu, x)(A) = \mu(xA)$ for measurable $A\subset E_k$, is continuous, the proof can be completed by the continuous mapping theorem. 
\end{proof}
\vspace{5pt}

\begin{proof}[Proof of Theorem \ref{t:consistency.Hill}]
Observe first that $\int_1^\infty \mu_k (B_t)t^{-1} \d t = (\alpha k-d)^{-1}$. 
By Proposition \ref{p:consistency.tail.measure} $(ii)$, it follows that  $m^{-k}\hat \nu_{k,n}(B_t) \stackrel{\Pr}{\to} \mu_k (B_t)$ for every $t>0$. Thus, the continuous mapping theorem yields that 
$$
\int_1^M m^{-k} \hat \nu_{k,n}(B_t) \frac{\d t}{t} \stackrel{\Pr}{\to} \int_1^M \mu_k (B_t)\frac{\d t}{t}, 
$$
for every $M>0$. According to \cite[Theorem 3.5]{Resnick:2007}, it suffices to demonstrate that for every $\delta>0$, 
$$
\lim_{M\to\infty} \limsup_{n\to\infty} \Pr \Big(  \int_M^\infty m^{-k}\hat \nu_{k,n}(B_t) \frac{\d t}{t} >\delta \Big) =0. 
$$
By Proposition \ref{p:consistency.tail.measure} $(ii)$, the probability above can be estimated as follows: 
\begin{align}
&\Pr \Big(  \int_M^\infty m^{-k}\hat \nu_{k,n}(B_t) \frac{\d t}{t} >\delta, \ \Big| \frac{\consisest}{R_k(C_kn/m)} - 1 \Big| <\frac{1}{2} \Big) + \Pr \Big( \Big| \frac{\consisest}{R_k(C_kn/m)} - 1 \Big| \ge\frac{1}{2} \Big) \label{e:tail.prob.negligible}\\
&\le \Pr \Big(  \int_{M/2}^\infty m^{-k}\nu_{k,n}(B_t) \frac{\d t}{t} >\delta \Big) + o(1) \notag \\
&\le \frac{1}{\delta}\, \int_{M/2}^\infty m^{-k} \E \big[ \nu_{k,n}(B_t) \big] \frac{\d t }{t} + o(1), \ \ \text{as } n\to\infty.   \notag
\end{align}

Following calculations similar to those for deriving \eqref{e:final.form.expectation}, along with Potter's bound in the Appendix (see also \eqref{e:Potter1} and \eqref{e:Potter2}), we can obtain that 
\begin{align*}
m^{-k}\E\big[ \nu_{k,n}(B_t) \big] &\le C^* \Big( \frac{n}{m} \Big)^k R_k(C_kn/m)^d f\big(R_k(C_k n/m)  \big)^k \int_t^\infty \rho^{-\alpha  +\vep + d -1 } \d \rho \\
& \le C^* t^{-\alpha +\vep+d}, 
\end{align*}
where $0<\vep <\alpha -d$ and the last inequality follows from \eqref{e:def.radius.function.R}. Now, \eqref{e:tail.prob.negligible} is upper bounded by 
$$
\frac{C^*}{\delta}\, \int_{M/2}^\infty t^{-\alpha +\vep+d-1} \d t \le \frac{C^*}{\delta} \Big(\frac{M}{2}\Big)^{-\alpha +\vep+d}, 
$$
and here, the last term clearly goes to $0$ as $M\to\infty$. 
\end{proof}
\medskip

\subsection{Proof of Asymptotic Normality}

We now present the proof of the results in Section \ref{sec:asym.normality}. First, we prove Proposition \ref{p:CLT.measure}, and the proof strategy can be summarized as follows.
For $K\ge1$, let 
 $\phi_k \in C_c(E_k)$ for $1\le k\le K$. We begin by truncating each random variable $\nu_{k,n}(\phi_k)$ in  a way that the truncated versions for different values of $k$ become independent. Using this independence, the multivariate asymptotic normality of the truncated random variables can be directly established from the univariate asymptotic normality of each component. For the required univariate asymptotic normality, Stein’s method for normal approximation (see Theorem 2.4 in \cite{Penrose:2003}) plays a critical role. Specifically, we follow an argument analogous to Proposition 7.3 in \cite{Owada:2017}. Finally, we establish the asymptotic normality of the original random variables by verifying that the truncation effect is negligible.

For every $k$-point vector $\bm{y} = (y_1, \dots, y_k) \in (\R^d)^k$,  define $\ms{max}(\bm{y}) := \max_{1\le i \le k}  |y_i|$, and for  every $T > 0$, set
\begin{equation}  \label{e:def.truncated.measure}
\nu_{k,n}^{(T)}(\cdot) := \frac{1}{k!} \sum_{\bm{y} \in (\Pn)_{\neq}^k} h_k(\bm{y}) \, \1\{ \ms{max}(\bm{y}) \le T \Rk \} \, \delta_{\bm{y}/ \Rk} (\cdot).
\end{equation}
In words, the process \eqref{e:def.truncated.measure} counts only  the $k$-point sets $\bm{y} \in (\Pn)_{\neq}^k$ that satisfy $h_k(\bm{y}) = 1$ and are contained in the ball of radius $T \Rk$ centered at the origin. 

The following lemma proves the multivariate asymptotic normality of the truncated process  \eqref{e:def.truncated.measure}. Before stating the lemma, we  introduce some notations to describe its limiting variance. For every  measurable functions $\phi_1, \phi_2$ on $E_k$,  define
\begin{align*}
V_k^{(T)}(\phi_1, \phi_2) := \begin{cases}
\frac{(\alpha k-d) D_{k,k}}{k!(C_k)^k} \int_{|x| \leq T} \phi_1(x, \dots, x)\phi_2(x, \dots, x) |x|^{-\alpha k} \d x \\[5pt]
&\hspace{-145pt} \text{ if } n f \lp \Rk \rp \to 0, \\[5pt]
\sum_{\ell = 1}^k \frac{\xi^{k-\ell} (\alpha k-d)D_{k, \ell}}{\ell! ( (k-\ell)! )^2(C_k)^k} \int_{|x| \leq T}\phi_1(x, \dots, x)\phi_2(x, \dots, x) |x|^{-\alpha (2k - \ell)} \d x \\[5pt]
&\hspace{-145pt} \text{ if } n f \lp \Rk \rp \to \xi \in (0, \infty), \\[5pt]
\frac{(\alpha k-d) D_{k,1}}{((k-1)! )^2(C_k)^k} \int_{|x| \leq T}\phi_1(x, \dots, x)\phi_2(x, \dots, x) |x|^{-\alpha (2k-1)} \d x \\[5pt]
&\hspace{-145pt} \text{ if } n f \lp \Rk \rp \to \infty, 
\end{cases}
\end{align*}
where  $D_{k, \ell}$ are the constants defined  in \eqref{e:def.Dkl}. 
It clearly holds that  $V_k^{(T)}(\phi_1, \phi_2) \to V_k(\phi_1, \phi_2)$ as $T\to\infty$. As before, we denote $V_k^{(T)}(\phi) = V_k^{(T)}(\phi, \phi)$. 

\begin{lemma}
\label{l:CLT.measure.trancated}
For every $K \ge 1$, let $\phi_k \in C_c(E_k)$ for $1\le k\le K$.   It then holds that 
$$
\Big(  \tau_{k,n}^{-1/2} \big( \nu_{k,n}^{(T)}(\phi_k) - \E [ \nu_{k,n}^{(T)}(\phi_k) ] \big) \Big)_{k = 1}^K \Rightarrow \big( \Phi_k^{(T)}(\phi_k) \big)_{k = 1}^K, \ \  \text{ as } n \to \infty, 
$$
where $(\tau_{k,n})_{n \ge 1}$ is given in \eqref{e:def.tau_kn} and 
$\big( \Phi_k^{(T)}\big)_{k=1}^K$ are independent and centered Gaussian random measures, such that the covariance kernel of $\Phi_k^{(T)}$ is given by   $V_k^{(T)}$. 
\end{lemma}

\begin{proof}[Proof of Lemma \ref{l:CLT.measure.trancated}]

The proof of the lemma is highly related to those of \cite[Proposition 7.3]{Owada:2017} and \cite[Theorem 3.9]{Penrose:2003}.
For each $k =1, \dots, K$, by an argument  identical to the proof of Lemma \ref{l:exp.var.asym} $(ii)$,  we can verify that 
$$
\tau_{k,n}^{-1} \V \big(   \nu_{k,n}^{(T)}(\phi_k)\big) \to V_k^{(T)} (\phi_k), \ \ \ n\to\infty. 
$$
Since each $\phi_k$ has compact support in $E_k$, there exists a constant $a_k>0$, such that the support of $\phi_k$ is contained in $B_{a_k}$, where $B_{a_k}$ is defined in \eqref{e:def.Br} with $r=a_k$. Assuming, without loss of generality, that $T>\max_{k=1,\dots,K}a_k$, we have for each $k=1,\dots,K$, 
\begin{align*}
\nu_{k,n}^{(T)}(\phi_k) &= \frac{1}{k!}  \sum_{\bm{y} \in (\Pn)_{\neq}^k} h_k(\bm{y})\, \phi_k \big( \Rk^{-1} \bm{y} \big)\, \\
& \qquad \qquad \times \1\big\{ \ms{max}(\bm{y}) \in [a_k \Rk, T\Rk] \big\} \\
&= \sum_{\Y\subset \Pn, \, |\Y|=k} h_k(\Y)\widetilde \phi_k \big(\Rk^{-1} \Y  \big)\\
& \qquad \qquad \times \1\big\{ \ms{max}(\Y) \in [a_k \Rk, T\Rk] \big\}, 
\end{align*}
where $\widetilde \phi_k$ is the symmetrization of $\phi_k$ (see \eqref{e:symmetrization}). 
Since   $R_j(t)/R_i(t) \to 0$ as $t\to\infty$, for each $1\le i < j \le K$, the random variables $\nu_{1,n}^{(T)}(\phi_1), \dots, \nu_{K,n}^{(T)}(\phi_K)$ are independent for sufficiently large  $n$.
Thus, the desired multivariate asymptotic normality follows if one can  prove that for each $k=1,\dots,K$, 
$$
\frac{\nu_{k,n}^{(T)}(\phi_k) - \E \big[\nu_{k,n}^{(T)}(\phi_k) \big]}{\sqrt{\V \big(  \nu_{k,n}^{(T)}(\phi_k)\big) }} \Rightarrow N(0, 1) \; \text{ as } n \to \infty. 
$$

Now, we exploit Stein's method based on dependency graph. Let $(Q_{\ell})_{\ell \in \N}$ be an enumeration of $d$-dimensional unit cubes covering $\R^d$ with $Q_{\ell}^o \cap Q_{\ell'}^o = \emptyset$ for any $\ell \neq \ell'$, where $Q_{\ell}^o$ denotes the interior of the cube $Q_{\ell}$. For each positive integer $n$, define
$$
W_n := \left\{ \ell \in \N  :  Q_{\ell} \cap \mathrm{Ann} \lp a_k \Rk, \, T \Rk \rp \neq \emptyset \right\},
$$
where $\mathrm{Ann}(a, b) := \{x \in \R^d  :  a \leq |x| \leq b \}$, $0 \le a <b$, 
is the closed annulus in $\R^d$. Notice that the cardinality of $W_n$ is upper bounded by $\Rk^d$ up to the scale. Then, $\nu_{k,n}^{(T)} (\phi_k)$ can be partitioned as follows:
\begin{align*}
\nu_{k,n}^{(T)}(\phi_k)  &= \sum_{\ell \in W_n} \sum_{\Y\subset \Pn, \, |\Y|=k} h_k(\Y)\, \widetilde \phi_k \big( \Rk^{-1} \Y \big)\\
&\qquad \qquad  \qquad \times  \1 \left\{ S(\Y) \in Q_\ell, \, \ms{max}(\Y) \in [a_k \Rk, T\Rk] \right\}  \\
&=: \sum_{\ell \in W_n} \eta_{\ell, k, n},
\end{align*}
where, for each $k$-point set $\Y = \{y_1, \dots, y_k\} \subset \R^d$, $S(\Y)$ denotes the element in $\Y$ satisfying $|S(\Y)| = \ms{max}(\Y)$. 
Now, let us define a graph $(W_n, \sim)$, in a way that  $W_n$ is the vertex set,  and   $i, j \in W_n$  are connected by an edge if and only if the distance between the cubes $Q_i$ and $Q_j$ are less than $2L$, where $L$ is a constant determined in \eqref{e:denseness.hk}. Then, the graph $(W_n, \sim)$ becomes a dependency graph with respect to the collection of random variables $(\eta_{\ell, k, n} )_{\ell \in W_n}$. 

Let $\Psi$ denote the distribution function of the standard Gaussian distribution on $\R$. Then, according to Stein's method for normal approximation regarding dependency graph (see, e.g., \cite[Theorem 2.4]{Penrose:2003}), it holds that, for every $z \in \R$, 
\begin{align}
&\left| \, \Pr \left( \frac{\nu_{k,n}^{(T)}(\phi_k) - \E \big[\nu_{k,n}^{(T)}(\phi_k) \big]}{\sqrt{\V \big(  \nu_{k,n}^{(T)}(\phi_k)\big) }}  \leq z \right) - \Psi(z) \,  \right|  \label{e:stein.normal.approx} \\[5pt]
&\le C^*\bigg\{  \sqrt{ R_k\Big( \frac{C_kn}{m} \Big)^d \max_{\ell \in W_n} \frac{\E\big[ \big| \eta_{\ell, k, n} - \E [\eta_{\ell, k, n} ] \big|^3\big]}{\big\{\V (  \nu_{k,n}^{(T)}(\phi_k))\big\}^{3/2}}} \notag \\
&  \quad + \quad \sqrt{ R_k\Big( \frac{C_kn}{m} \Big)^d \max_{\ell \in W_n}  \frac{\E\big[ \big( \eta_{\ell, k, n} - \E [\eta_{\ell, k, n} ] \big)^4\big]}{\big\{\V (  \nu_{k,n}^{(T)}(\phi_k))\big\}^2} } \bigg\}. \notag
\end{align}
Hence, the proposed result follows if we can show that  the right-hand side in \eqref{e:stein.normal.approx} tends to zero as $n \to \infty$ in all three possible scenarios depending on the limit value of $n f \lp \Rk \rp$. The remainder of the proof is very similar to that of \cite[Proposition 7.3]{Owada:2017}; hence, we omit it here.
\end{proof}

\begin{remark}
For a particular choice of the cut-off sequence $m=m(n)$, Lemma \ref{l:CLT.measure.trancated} can be strengthened using more recent techniques on normal approximation for $U$-statistics of Poisson processes (\cite{Reitzner;Schulte:2013, Schulte:2016}). For example, applying \cite[Theorem 6.2]{Reitzner;Schulte:2013}, our calculation shows that the Wasserstein distance between $\big\{\text{Var}(\nu_{k,n}^{(T)}(\phi_k)\big\}^{-1/2} \big( \nu_{k,n}^{(T)}(\phi_k) - \E[ \nu_{k,n}^{(T)}(\phi_k)] \big)$
 and the standard normal distribution is bounded, up to constants, by  
\begin{equation}  \label{e:Malliavin.Stein.bdd}
\frac{n^{-1/2} (n/m)^{k/2}}{f \lp R_k(C_k n/m) \rp^{k-1}}.
\end{equation}
When $m=m(n)$ grows quickly enough, \eqref{e:Malliavin.Stein.bdd} tends to $0$, yielding a stronger conclusion than Lemma \ref{l:CLT.measure.trancated}, namely, convergence in Wasserstein distance. However, if $m=m(n)$ does not grow fast enough, \eqref{e:Malliavin.Stein.bdd} may fail to converge to $0$. For this reason, and to avoid putting additional conditions on the rate of $m(n)$, we have chosen to rely on Lemma \ref{l:CLT.measure.trancated} rather than the approach in \cite{Reitzner;Schulte:2013, Schulte:2016}.
\end{remark}

\begin{proof}[Proof of Proposition \ref{p:CLT.measure}]
The rest of our argument is dedicated to showing  that  the truncation in \eqref{e:def.truncated.measure} is asymptotically negligible. Note first that we have 
$$
\Phi_k^{(T)} (\phi_k) \Rightarrow \Phi_k (\phi_k), \ \ \text{as  } T\to\infty, 
$$
because $V_k^{(T)}(\phi_k) \to V_k(\phi_k)$ as $T\to\infty$. For every $k=1,\dots,K$, we define 
$$
\bar \nu_{k,n}(\phi_k) := \nu_{k,n}(\phi_k) - \E \big[ \nu_{k,n}(\phi_k) \big]. 
$$
Similarly, one can define the truncated version $\bar \nu_{k,n}^{(T)}(\phi_k)$ in an analogous manner.

According to \cite[Theorem 3.5]{Resnick:2007}, it now suffices to prove that  for every   $k=1,\dots,  K$, and $\vep > 0$, 
\begin{equation}  \label{e:second.converging}
\lim_{T \to \infty} \limsup_{n \to \infty} \Pr \left(  \tau_{k,n}^{-1/2} \big| \bar \nu_{k,n}(\phi_k) - \bar \nu_{k,n}^{(T)}(\phi_k)  \big|
>\vep \right) = 0.
\end{equation}
By Chebyshev's inequality, we have 
\begin{align}
&\Pr \Big(  \tau_{k,n}^{-1/2} \big| \bar \nu_{k,n}(\phi_k) - \bar \nu_{k,n}^{(T)}(\phi_k)  \big| >\vep \Big)  \label{e:variance.with.indicator} \\
&\le \frac{1}{\vep^2 \tau_{k,n}}\, \V \Big(  \sum_{\Y\subset \Pn, \, |\Y|=k} h_k(\Y) \, \widetilde \phi_k \big( \Rk^{-1}\Y \big)\, \1\{ \ms{max}(\Y) > T \Rk \}  \Big). \notag 
\end{align}
The variance  above has essentially the same structure as in \eqref{e:nukn.no.permutation.expression}, with the only difference being the additional indicator in \eqref{e:variance.with.indicator}. Hence, by repeating the calculations from the proof of Lemma \ref{l:exp.var.asym} $(ii)$, the last term in \eqref{e:variance.with.indicator} converges to $\vep^{-2}\big( V_k(\phi_k) - V_k^{(T)}(\phi_k)\big)$, which however vanishes as $T\to\infty$. Now, \eqref{e:second.converging} has been verified. 
\end{proof}

Next we prove the functional-level asymptotic normality for the process $\big(\nu_{k,n}(t), \, t\ge0\big)$ in \eqref{e:nu.kn.proc.version}. 

\begin{proof}[Proof of Proposition \ref{p:CLT.functional}]
Our proof is divided into two parts. First, we prove  the finite-dimensional weak convergence. 
From Proposition \ref{p:CLT.measure}, together with the Cram\'er--Wold theorem, it follows that 
$$
\lp \tau_{k,n}^{-1/2} \lp \nu_{k,n}(t) - \E [ \nu_{k,n}(t) ] \rp \rp_{k = 1}^K \stackrel{f.d.d}{\Rightarrow} \lp \Phi_k(B_t) \rp_{k = 1}^K, \ \  \text{ as } n \to \infty, 
$$
where   $\overset{f.d.d}{\Rightarrow}$ means  finite-dimensional weak convergence. Thus it suffices to identify the covariance of the  process $\lp \Phi_k(B_t),\,  t \in (0, \infty] \rp$ for each $k$. By \eqref{e:def.Vk}, one can see that for $s,t \in (0,\infty]$, 
\begin{align*}
&\Cov \big( \Phi_k(B_t), \Phi_k (B_s) \big) = \begin{cases}
L_{k,k} (t\vee s)^{d-\alpha k} & \hspace{-40pt} \text{if } nf\big( \Rk \big)\to 0, \\[3pt]
\sum_{\ell=1}^k \xi^{k-\ell} L_{k,\ell} (t\vee s)^{d-\alpha (2k-\ell)} \\
& \hspace{-40pt} \text{if } nf\big( \Rk \big) \to \xi \in (0, \infty), \\[3pt]
L_{k,1} (t\vee s)^{d-\alpha (2k-1)} & \hspace{-40pt} \text{if } nf\big( \Rk \big) \to \infty. 
\end{cases}
\end{align*}
It is straightforward to verify that the above expression coincides with $\Cov\big( W_k(t), W_k(s) \big)$ regardless of the limit value of $nf\big( \Rk \big)$. This allows us to conclude the proof of the finite-dimensional weak convergence.

The remainder of the discussion  is focused on showing  the tightness.  According to criteria for the tightness in the space $\mathcal{D} (0, \infty]$, which are  given, for example, in  \cite[Theorem 13.5]{Billingsley:1999}, it suffices to prove the following: with a fixed $\delta > 0$, there exist  constants $A > 0$, $N \in \N$, and $q > 1$, such that for all $\delta \leq r \leq s \leq t \leq \infty$ and $n \geq N$, 
\begin{equation}  \label{e:sufficient.cond.tightness}
\tau_{k,n}^{-2} \E \lb \lp \gamma_{k, n,s,t} - \E[\gamma_{k, n,s,t}] \rp^2 \lp \gamma_{k, n, r,s} - \E[\gamma_{k, n,r,s}] \rp^2 \rb \leq A(t^d - r^d)^{1 + 1/q},
\end{equation}
where $\gamma_{k, n,s,t} := \nu_{k,n}(s) -\nu_{k,n}(t)$ for $\delta \le s \le t\le \infty$. 

The criterion \eqref{e:sufficient.cond.tightness} is slightly modified from the version in \cite[Theorem 13.5]{Billingsley:1999}. Specifically, we consider the time-reversed space $\mathcal D(0,\infty]$, whereas the original criterion in \cite{Billingsley:1999} is formulated for $\mathcal D[0,\infty)$. According to the weak convergence theory in $\mathcal D[0,\infty)$ (see \cite[Chapter 16]{Billingsley:1999}), tightness in $\mathcal D[0,\infty)$ follows from the criterion in \cite[Theorem 13.5]{Billingsley:1999} applied to  $\mathcal D[0,L]$ for each $L>0$. Analogously, in our time-reversed setting it suffices to show tightness in $\mathcal D[\delta,\infty]$ for each $\delta>0$. Moreover, the exponent on the right-hand side of \eqref{e:sufficient.cond.tightness} is defined slightly differently from that in \cite[Theorem 13.5]{Billingsley:1999}, for our convenience. 


To begin, we  introduce additional notations. For any $\delta \leq s \leq t \leq \infty$, define 
$$
h_{k, n, s, t}(\mathcal{Y}) := h_k(\mathcal{Y}) \1\{s \leq \Rk^{-1} \ms{min}(\mathcal{Y}) < t \}, \ \ \ \Y=\{ y_1,\dots,y_k \} \subset  \R^d,  \ \ y_i\in \R^d, 
$$
so that 
$$
\gamma_{k, n, s, t} = \sum_{\mathcal{Y} \subset \mathcal{P}_n, |\mathcal{Y}| = k} h_{k, n, s, t}(\mathcal{Y}).
$$
Now, the left-hand side in \eqref{e:sufficient.cond.tightness} can be denoted as 
\begin{equation}  \label{e:tightness.combinatrics}
\tau_{k,n}^{-2} \sum_{p=0}^2 \sum_{q = 0}^2 \binom{2}{p} \binom{2}{q} (-1)^{p + q} F_n(p, q),
\end{equation}
where 
$$
F_n(p, q) := \E \big[ \gamma_{k, n,s,t}^p \gamma_{k, n,r,s}^q \big] \lp \E [ \gamma_{k, n,s,t}] \rp^{2- p} \lp \E [ \gamma_{k, n, r, s} ] \rp^{2 - q}.
$$
Note that, for any $p, q \in \{0, 1, 2\}$, we may write
$$
F_n(p, q) = \E \lb \sum_{\mathcal{Y}_1 \subset \mathcal{P}_n^{(1)}} \hspace{-2.5pt} \sum_{\mathcal{Y}_2 \subset \mathcal{P}_n^{(2)}} \hspace{-2.5pt} \sum_{\mathcal{Y}_3 \subset \mathcal{P}_n^{(3)}} \hspace{-2.5pt} \sum_{\mathcal{Y}_4 \subset \mathcal{P}_n^{(4)}} \hspace{-4pt} h_{k, n,s,t}(\mathcal{Y}_1) h_{k, n,s,t}(\mathcal{Y}_2) h_{k, n,r,s}(\mathcal{Y}_3) h_{k, n,r,s}(\mathcal{Y}_4) \rb,
$$

where for every pair $i \neq j$, either $\mathcal{P}_n^{(i)} = \mathcal{P}_n^{(j)}$ or $\mathcal{P}_n^{(i)}$ and $\mathcal{P}_n^{(j)}$ are independent copies of each other.

By following the same reasoning as in the proof of \cite[Theorem 4.3]{Owada:2017}, a careful examination of the expression  \eqref{e:tightness.combinatrics} reveals that many of the terms in \eqref{e:tightness.combinatrics} will cancel each other. Consequently, it remains to demonstrate that 
\begin{align}
\tau_{k,n}^{-2} & \E \bigg[ \sum_{\mathcal{Y}_1 \subset \mathcal{P}_n} \sum_{\mathcal{Y}_2 \subset \mathcal{P}_n} \sum_{\mathcal{Y}_3 \subset \mathcal{P}_n} \sum_{\mathcal{Y}_4 \subset \mathcal{P}_n} h_{k,n,s,t}(\mathcal{Y}_1) h_{k,n,s,t}(\mathcal{Y}_2) h_{k,n,r,s}(\mathcal{Y}_3) h_{k,n,r,s}(\mathcal{Y}_4) \label{e:cross.term.considered} \\
& \hspace{-6pt} \times \prod_{i=1}^4 \1\{ \mathcal{Y}_i \text{ shares at least one common element with at least one of the other } \Y_j, \, j\neq i \} \bigg] \notag 
\end{align}
is bounded above by $(t^d - r^d)^{1 + 1/q}$, up to constants. For this purpose, 
We consider the following four cases:

\begin{itemize}

\item[($i$)] $\mathcal{Y}_1 \cap \mathcal{Y}_2 \neq \emptyset$ and $\mathcal{Y}_3 \cap \mathcal{Y}_4 \neq \emptyset$ with $(\mathcal{Y}_1 \cup \mathcal{Y}_2) \cap (\mathcal{Y}_3 \cup \mathcal{Y}_4) = \emptyset$, 

\item[($ii$)] $\mathcal{Y}_1 \cap \mathcal{Y}_3 \neq \emptyset$ and $\mathcal{Y}_2 \cap \mathcal{Y}_4 \neq \emptyset$ with $(\mathcal{Y}_1 \cup \mathcal{Y}_3) \cap (\mathcal{Y}_2 \cup \mathcal{Y}_4) = \emptyset$, 

\item[($iii$)] $\mathcal{Y}_1 \cap \mathcal{Y}_4 \neq \emptyset$ and $\mathcal{Y}_2 \cap \mathcal{Y}_3 \neq \emptyset$ with $(\mathcal{Y}_1 \cup \mathcal{Y}_4) \cap (\mathcal{Y}_2 \cup \mathcal{Y}_3) = \emptyset$, 

\item[($iv$)] Each $\Y_i$ shares at least one common element with at least one of the other $\Y_j$, $j\neq i$, but none of the conditions $(i)$--$(iii)$ holds. 
\end{itemize}

Since Cases $(i)$, ($ii$), and ($iii$) can be handled in the same way,  we analyze  Case ($i$) only. 
Specifically, letting $\ell, \ell' \in \{1, \dots, k\}$, we have to find an upper bound of 
\begin{align*}
& \tau_{k,n}^{-1} \E \lb \sum_{\mathcal{Y}_1 \subset \mathcal{P}_n} \sum_{\mathcal{Y}_2 \subset \mathcal{P}_n} h_{k, n,s, t}(\mathcal{Y}_1) h_{k, n, s, t}(\mathcal{Y}_2) \1 \{ |\mathcal{Y}_1 \cap \mathcal{Y}_2| = \ell \} \rb \\
& \times \tau_{k,n}^{-1} \E \lb \sum_{\mathcal{Y}_3 \subset \mathcal{P}_n} \sum_{\mathcal{Y}_4 \subset \mathcal{P}_n} h_{k, n, r, s}(\mathcal{Y}_3) h_{k, n, r,s}(\mathcal{Y}_4) \1\{ |\mathcal{Y}_3 \cap \mathcal{Y}_4| = \ell' \} \rb =: A_1 \times A_2.
\end{align*}
By the multivariate Mecke formula, as well as the same change of variables as in \eqref{e:first.change.of.variables}--\eqref{e:final.form.expectation}, the term $A_1$ can be written as 
\begin{align}
\begin{split}  \label{e:final.form.A1}
A_1 & =  \frac{\tau_{k,n}^{-1} n^{2k - \ell}}{\ell ! ((k-\ell)!)^2}\, \Rk^d f \lp \Rk  \rp^{2k - \ell} \\
&\:  \times \int_{\rho = 0}^{\infty} \int_{\theta \in \mathbb{S}^{d-1}} \int_{\bz \in (\R^d)^{k-1}}  h_k(0, z_1, \dots, z_{k-1}) h_k(0, z_1, \dots, z_{\ell-1}, z_k, \dots, z_{2k - \ell -1}) \\
&\: \times \1\big\{ \ms{min} (\rho \theta, \rho \theta + \Rk^{-1} z_1,\dots, \rho \theta + \Rk^{-1} z_{2k - \ell - 1}) \in [s, t) \big\} \\
&\:  \times \frac{f \lp \Rk \rho  \rp}{f \lp \Rk \rp} \prod_{i=1}^{2k - \ell - 1} \frac{f \lp \Rk |\rho \theta + \Rk^{-1} z_i| \rp}{f \lp \Rk \rp} \rho^{d-1} \d \bz \d \sigma(\theta) \d \rho.
\end{split}
\end{align}
By the definition of $\tau_{k,n}$ in \eqref{e:def.tau_kn}, the factor $\tau_{k,n}^{-1} n^{2k - \ell} \Rk^d f \lp \Rk \rp^{2k - \ell}$ is bounded by a finite and positive constant. Further, the indicator function in \eqref{e:final.form.A1} can be bounded as follows. 
\begin{align*}
&\1 \{ \ms{min}(\rho \theta, \rho \theta + \Rk^{-1} z_1, \dots, \rho \theta + \Rk^{-1} z_{2k - \ell - 1}) \in [s, t) \} \\
&\quad = \1 \left\{ \bigvee_{i=0}^{2k - \ell - 1}  |\rho \theta + \Rk^{-1} z_i| < t  \right\} \prod_{i=1}^{2k - \ell - 1} \1 \big\{ |\rho \theta + \Rk^{-1} z_i | \geq s \big\} \\
&\quad \le  \sum_{i=0}^{2k - \ell - 1} \1 \big\{ r\le | \rho \theta + \Rk^{-1} z_i | < t \big\}, 
\end{align*}
with $z_0 \equiv 0$. 
By Potter's bounds, for any $0 < \vep < \alpha -d$, we  obtain that
$$
\frac{f \lp \Rk \rho  \rp}{f \lp \Rk  \rp} \1 \{ \rho \geq s \} \leq ( 1+ \vep) \rho^{-\alpha + \vep} \1 \{ \rho \geq \delta \}
$$
for  sufficiently large  $n$. Similarly, for every $i=1,\dots,2k-\ell-1$, and sufficiently large $n$,
$$
\frac{f \lp \Rk | \rho \theta + \Rk^{-1} z_i|  \rp}{f \lp \Rk \rp} \1 \big\{ | \rho \theta + \Rk^{-1} z_i| \geq s \big\} \le (1 + \vep) \delta^{-\alpha + \vep}. 
$$
Substituting all these bounds, \eqref{e:final.form.A1} is now  upper bounded by 
\begin{align*}
& C^* (1 + \vep)^{2k - \ell} \delta^{-(2k - \ell - 1)(\alpha - \vep)} \\
&\qquad  \qquad  \times \sum_{i=0}^{2k - \ell - 1} \int_{\bz \in (\R^d)^{2k - \ell - 1}} h_k(0, z_1, \dots, z_{k-1})\,  h_k(0, z_1, \dots, z_{\ell - 1}, z_k, \dots, z_{2k - \ell -1}) \\
&\qquad  \qquad \times \int_{\rho = 0}^{\infty} \int_{\theta \in \mathbb{S}^{d-1}}  \rho^{d-1 - \alpha + \vep}\, \1 \big\{ r \le | \rho \theta + \Rk^{-1} z_i |< t \big\}   \d \sigma(\theta) \d \rho \d \bz. 
\end{align*}
Performing change of variables  $v=\rho\theta + \Rk^{-1}z_i$, the inner integral over $(\rho, \theta)$ above can be estimated as 
\begin{align*}
&\int_{\rho = 0}^{\infty} \int_{\theta \in \mathbb{S}^{d-1}} \rho^{d-1-\alpha +\vep} \1 \big\{ r \le | \rho \theta + \Rk^{-1} z_i |< t \big\}  \d \sigma(\theta) \d \rho \\
&=\int_{r \le |v| \le t} \Big| v-\frac{z_i}{\Rk} \Big|^{-\alpha +\vep} \d v  \le \kappa_d \Big( \delta - \frac{L}{\Rk} \Big)^{-\alpha+\vep} (t^d-r^d). 
\end{align*} 
Now, it can be concluded  that $A_1$ in \eqref{e:final.form.A1} is bounded above by $ (t^d - r^d)$ up to constants. Similarly,  $A_2$ is also  bounded by $(t^d - r^d)$ up to constants. This completes the discussion for Case $(i)$. 

Next let us consider Case $(iv)$. For brevity, let $\mathcal{E}$ denote the event that  $(\Y_1, \dots, \Y_4)$ satisfies the conditions in  Case $(iv)$.  Note that 
\begin{align*}
\mE &= \mE \cap \big\{ (\Y_1 \cup \Y_2) \cap (\Y_3 \cup \Y_4) \neq \emptyset \big\} \\
&\subset \big( \mE \cap \{ \Y_1\cap \Y_3 \neq \emptyset \}  \big) \cup \big( \mE \cap \{ \Y_1\cap \Y_4 \neq \emptyset \}  \big) \\
& \qquad \qquad \qquad \qquad  \cup \big( \mE \cap \{ \Y_2\cap \Y_3 \neq \emptyset \}  \big) \cup \big( \mE \cap \{ \Y_2\cap \Y_4 \neq \emptyset \}  \big) \\
&=: \bigcup_{i=1}^4 \mE_i, 
\end{align*}
and hence, $\1_\mE\le \sum_{i=1}^4 \1_{\mE_i}$. 

Below, we only find an upper bound of \eqref{e:cross.term.considered} under the event $\mE_1$, because the other three cases for  $\mE_2$, $\mE_3$, and $\mE_4$ can be  treated  in the same way. More specifically, letting $\ell' \in \{1,\dots,k\}$ and $\ell\in \{ 3,\dots,3k \}$, we need to bound 
\begin{align}
\tau_{k,n}^{-2} & \E \bigg[ \sum_{\mathcal{Y}_1 \subset \mathcal{P}_n} \sum_{\mathcal{Y}_2 \subset \mathcal{P}_n} \sum_{\mathcal{Y}_3 \subset \mathcal{P}_n} \sum_{\mathcal{Y}_4 \subset \mathcal{P}_n} h_{k,n,s,t}(\mathcal{Y}_1) h_{k,n,s,t}(\mathcal{Y}_2) h_{k,n,r,s}(\mathcal{Y}_3) h_{k,n,r,s}(\mathcal{Y}_4)  \label{e:under.E1} \\
&\qquad \times \1_{\mE_1} \times \1 \big\{ |\mathcal{Y}_1 \cap \mathcal{Y}_3| = \ell', \  | \mathcal{Y}_1 \cup \mathcal{Y}_2 \cup \mathcal{Y}_3 \cup \mathcal{Y}_4| = 4k - \ell \big\} \bigg].\notag
\end{align}
Observe that if $\ell' = k$, then $\Y_1 = \Y_3$, and the expectation above becomes non-zero only when $\Rk^{-1} \ms{min}(\Y_1) \in [s,t) \cap [r,s)$; however, this condition cannot be satisfied. Similarly, the case $\ell = 3k$ can be excluded for the same reason. Therefore, in the following, we may assume  $\ell' \in \{1, \dots, k-1\}$ and $\ell \in \{3, \dots, 3k-1\}$.

Applying the multivariate Mecke formula, up to the scale, \eqref{e:under.E1}  can be written as
\begin{align*}
\tau_{k,n}^{-2} n^{4k - \ell} & \int_{\bx \in (\R^d)^{4k - \ell}} h_{k, n,s,t}(x_1, \dots, x_k) h_{k, n,r,s}(x_1, \dots, x_{\ell'}, x_{k +1}, \dots, x_{2k - \ell'}) \\
& \times h_{k, n,s,t}(\bx^{(1)}) h_{k,n,r,s}(\bx^{(2)}) \prod_{i=1}^{4k - \ell} f(x_i) \d \bx,
\end{align*}
where $\bx^{(1)}$ and $\bx^{(2)}$ are vectors of length $k$  consisting of some parts of the vector $\bx$ of length $4k - \ell$,  such that $\bx= (x_1,\dots, x_{2k-\ell'}) \cup \bx^{(1)}\cup \bx^{(2)}$. 
By the same change of variables as in the proof of Lemma \ref{l:exp.var.asym}, the above is equal to 
\begin{align}
\begin{split}  \label{e:final.form.tightness}
&\tau_{k,n}^{-2} n^{4k - \ell} \Rk^d f \lp \Rk \rp^{4k - \ell} \\
&\quad \times \int_{\rho = 0}^{\infty} \int_{\theta \in \mathbb{S}^{d-1}} \int_{\bz \in (\R^d)^{4k - \ell - 1}} \hspace{-10pt} h_k(0, z_1, \dots, z_{k-1}) h_k(0, z_1, \dots, z_{\ell'-1}, z_k, \dots, z_{2k - \ell' - 1}) \\
& \quad \times h_k(\bz^{(1)}) h_k(\bz^{(2)}) \\
&\quad \times \1\{ \ms{min}(\rho \theta, \rho \theta + \Rk^{-1} z_1, \dots, \rho \theta + \Rk^{-1} z_{k-1}) \in [s, t) \} \\
&\quad \times \1\{ \ms{min}(\rho \theta, \rho \theta + \Rk^{-1} z_1, \dots, \rho \theta + \Rk^{-1} z_{\ell' - 1}, \\
&\qquad \qquad \qquad  \qquad \qquad \rho \theta + R_k(n/m)^{-1} z_k, \rho \theta + R_k(n/m)^{-1} z_{2k - \ell' - 1}) \in [r, s) \} \\
&\quad \times \1\{ \ms{min}(\rho \theta + \Rk^{-1} \bz^{(1)}) \in [s, t) \} \, \1\{ \ms{min}(\rho \theta + \Rk^{-1} \bz^{(2)}) \in [r, s) \} \\
&\quad \times \frac{f \lp \Rk \rho \rp}{f \lp \Rk  \rp} \prod_{i=1}^{4k - \ell - 1} \frac{f \lp \Rk |\rho \theta + \Rk^{-1} z_i|  \rp}{f \lp \Rk  \rp} \rho^{d-1} \d \bz \d \sigma(\theta) \d \rho.
\end{split}
\end{align}
Now, we shall find the upper bounds for each of the terms in \eqref{e:final.form.tightness}. First, 
it follows from the Potter bounds that under the condition on the indicator functions,  
$$
\begin{aligned}
& \frac{f \lp \Rk \rho  \rp }{f \lp \Rk \rp} \prod_{i=1}^{4k - \ell - 1} \frac{f \lp \Rk | \rho \theta + \Rk^{-1} z_i|  \rp }{f \lp \Rk  \rp}  \\
& \leq (1 + \vep)^{4k-\ell}  \delta^{-(4k-\ell-1)(\alpha -\vep)}\rho^{-\alpha + \vep}, 
\end{aligned}
$$
where $0<\vep < \alpha -d$. 
Moreover, we observe that the product of the first two indicator functions in \eqref{e:final.form.tightness} can be bounded by  
$$
\sum_{i=k}^{2k - \ell' - 1} \sum_{j=0}^{k-1} \1\{r\le  |\rho \theta + \Rk^{-1} z_i| <s, \ s\le  | \rho \theta + \Rk^{-1} z_j| <t \}.
$$
Further, due to the restriction on the event $\mE_1$ and the condition \eqref{e:denseness.hk} for $h_k$, each element in $\bz^{(1)}$ and $\bz^{(2)}$ must be bounded  from the origin.

Now, combining these observations with the derived bounds, while bounding $h_k(\bz^{(1)})$, $h_k(\bz^{(2)})$, and their corresponding indicator functions by $1$ respectively,  it suffices to estimate the following term for each $i\in \{ k,\dots,2k-\ell'-1 \}$ and $j\in \{ 0,\dots,k-1 \}$: 
\begin{align}
C_{ij} & := \tau_{k,n}^{-2} n^{4k - \ell} \Rk^d f \lp \Rk \rp^{4k - \ell} \label{e:C_ij} \\
&\qquad \times \int_{\rho = 0}^{\infty} \int_{\theta \in \mathbb{S}^{d-1}} \int_{\bz \in (\R^d)^{2k - \ell'-1}} h_k(0, z_1, \dots, z_{k-1}) h_k(0, z_1,\dots,z_{\ell'-1}, z_k, \dots, z_{2k - \ell'-1}) \notag \\
&\qquad \times  \1\big\{ r \leq |\rho \theta + \Rk^{-1} z_i| < s, \ s \leq |\rho \theta + \Rk^{-1} z_j| < t \big\} \notag \\
& \qquad \times \rho^{d-1-\alpha+\vep} \d \bz \d \sigma(\theta) \d \rho \notag \\
&= \tau_{k,n}^{-2} n^{4k - \ell} \Rk^d f \lp \Rk \rp^{4k - \ell} \notag \\
&\qquad  \times \int_{v\in \R^d} \int_{\bz\setminus \{ z_i \} \in (\R^d)^{2k-\ell'-2}} h_k(0,z_1,\dots,z_{k-1}) \, \notag \\
& \qquad \qquad \qquad \qquad \qquad \qquad \times |v|^{-\alpha+\vep} \, \1 \big\{  s \le |v+\Rk^{-1}z_j|<t \big\} \notag  \\
&\qquad \times \Big(  \int_{z_i \in \R^d} h_k(0, z_1,\dots,z_{\ell'-1}, z_k, \dots, z_{2k - \ell'-1})\, \notag \\
&\qquad \qquad \qquad \qquad \qquad \qquad \times \1\big\{ r\le |v+\Rk^{-1}z_i| <s  \big\}\d z_i \Big) \d (\bz\setminus \{z_i\}) \d v. \notag
\end{align}

Let $I$ denote the inner integral with respect to $z_i$. For every $v\in \R^d$ and $\bz \setminus \{z_i\} \in (\R^d)^{2k - \ell'-2}$, H\"{o}lder's inequality gives us 
\begin{align*}
I &\leq \Big\{ \int_{\R^d} h_k(0, z_1, \dots, z_{\ell'-1}, z_k, \dots, z_{2k - \ell'-1}) \d z_i \Big\}^{1/p} \\
& \qquad \times \Big\{ \int_{\R^d} \1\{ r \leq |y + \Rk^{-1} z_i| < s\} \d z_i \Big\}^{1/q} \\
&\le \Big\{ \int_{\R^d} h_k(0, z_1, \dots, z_{\ell'-1}, z_k, \dots, z_{2k - \ell'-1}) \d z_i \Big\}^{1/p} \kappa_d^{1/q} \Rk^{d/q} (t^d-r^d)^{1/q}, 
\end{align*}
for every $p, q > 1$ with $1/p + 1/q = 1$. 
Substituting this  bound into \eqref{e:C_ij} and noting that 
$$
\int_{\R^d} |v|^{-\alpha+\vep} \1 \big\{ s\le |v+\Rk^{-1}z_j| <t \big\}\d v \le \Big(  \delta - \frac{L}{\Rk}\Big)^{-\alpha +\vep} \kappa_d (t^d-r^d), 
$$
we can obtain 
\begin{align}
C_{i,j} &\le C^* \tau_{k,n}^{-2} n^{4k - \ell} \Rk^{(1+1/q)d} f \lp \Rk \rp^{4k - \ell} (t^d-r^d)^{1+1/q} \label{e:Cij.upper.bound} \\
&\le C^* \tau_{k,n}^{-2} m^{k(1+1/q)} \big( nf(\Rk) \big)^{(3-1/q)k-\ell} (t^d-r^d)^{1+1/q}. \notag
\end{align}
For the last inequality above, we have used \eqref{e:def.radius.function.R}. 

If $nf(\Rk) \to 0$ or $\xi \in (0, \infty)$, then $\tau_{k,n} = m^k$, and \eqref{e:Cij.upper.bound} is maximized when $\ell = 3k - 1$, yielding an upper bound of
$$
C^* m^{-k(1 - 1/q)} \big( nf(\Rk) \big)^{1 - k/q} (t^d - r^d)^{1 + 1/q}.
$$
If we choose $q > k$, the last term can be upper bounded by $(t^d - r^d)^{1 + 1/q}$, up to a scaling factor. If $nf(\Rk) \to \infty$, then $\tau_{k,n} = \big( nf(\Rk) \big)^{k - 1}$, and \eqref{e:Cij.upper.bound} is maximized at $\ell = 3$. In this case, \eqref{e:Cij.upper.bound} is bounded by
$$
C^* m^{-k(1 - 1/q)} \big( nf(\Rk) \big)^{k(1 - 1/q) - 1} (t^d - r^d)^{1 + 1/q}.
$$
Here, it is sufficient to choose $q \in (1, k/(k-1))$, and the final term is, once again, bounded by $(t^d - r^d)^{1 + 1/q}$, up to constants.
\end{proof}

To prove Theorem \ref{t:CLT.Hill.estimator}, we need to show that the asymptotic normality  in Proposition \ref{p:CLT.functional} is preserved under the map $x \mapsto \int_{1}^{\infty} x(t) t^{-1} \d t$, after replacing the unknown radius $\Rk$ with its consistent estimator $\consisest$. 

\begin{proof}[Proof of Theorem \ref{t:CLT.Hill.estimator}]
From Propositions \ref{p:consistency.tail.measure} and \ref{p:CLT.functional}, we have  that as $n\to\infty$, 
$$
\begin{aligned}
& \lp \lp \tau_{k,n}^{-1/2} \lp \nu_{k,n}(t) - \E \lb \nu_{k,n}(t) \rb \rp, \,  t \in (0, \infty] \rp, \, \frac{\consisest}{\Rk} \rp_{k = 1}^K  \\
& \Rightarrow \lp \lp W_{k}(t), \, t \in (0, \infty] \rp, 1 \rp_{k = 1}^K
\end{aligned}
$$
in the space $\lp \mathcal{D}(0, \infty] \times (0,\infty) \rp^K$.
As an analogue of \eqref{e:nu.kn.proc.version}, we define 
$$
\hat \nu_{k,n} (t) := \hat \nu_{k,n}(B_t), \ \ \ t\in (0,\infty]. 
$$
For each $1 \le k \le K$, define the map $J : \mathcal{D}(0, \infty] \times (0,\infty) \to \mathcal{D}(0, \infty]$ by $J \lp x(\cdot), r \rp := x (r \cdot)$ such that 
$$
J \lp \nu_{k,n}(\cdot ), \frac{\consisest}{\Rk} \rp (t) =  \hat{\nu}_{k,n}(t), \ \  \ t\in (0,\infty], 
$$
Since $J$ is continuous on the support of the process $W_k$, it follows from the continuous mapping theorem that,  as $n\to\infty$, 
\begin{align*}
&\lp \lp \tau_{k,n}^{-1/2} \lp \hat{\nu}_{k,n}(t) - \E \lb \nu_{k,n}(s) \rb \Big|_{s = \frac{\consisest t}{\Rk}} \rp, \, t \in (0, \infty] \rp \rp_{k = 1}^K \\
&\qquad \Rightarrow \lp\lp W_{k}(t), \,  t \in (0, \infty] \rp\rp_{k =1}^K, 
\end{align*}
in the space $\big(\mathcal{D} (0, \infty]\big)^K$.

Since the map $x(\cdot) \mapsto \int_{1}^M x(t) t^{-1} \d t$ is continuous on $\mathcal{D}(0, \infty]$ for every finite $M > 0$, the desired result follows if we can show that, for every   $k =1,\dots,K$, and $\delta >0$ 
\begin{equation}  \label{e:second.converging.Hill.asym.normality}
\lim_{M \to \infty} \limsup_{n \to \infty} \Pr \bigg(  \tau_{k,n}^{-1/2} \bigg| \int_{M}^{\infty} \hat{\nu}_{k,n}(t) \frac{\d t}{t} - \int_{M}^{\infty} \E \lb \nu_{k,n}(s) \rb\Big|_{s = \frac{\consisest t}{\Rk} } \frac{\d t}{t} \bigg| > \delta \bigg) = 0; 
\end{equation}
see \cite[Theorem 3.5]{Resnick:2007}. 
Using an argument similar to  \eqref{e:tail.prob.negligible}, along with Chebyshev's inequality, 
\begin{align*}
&\limsup_{n \to \infty} \Pr \bigg(  \tau_{k,n}^{-1/2} \bigg| \int_{M}^{\infty} \hat{\nu}_{k,n}(t) \frac{\d t}{t} - \int_{M}^{\infty} \E \lb \nu_{k,n}(s) \rb\Big|_{s = \frac{\consisest t}{\Rk} } \frac{\d t}{t} \bigg| > \delta \bigg)  \\
&\le \limsup_{n \to \infty} \Pr \bigg(  \tau_{k,n}^{-1/2}\int_{M/2}^\infty \big| \nu_{k,n}(t) -\E [\nu_{k,n}(t)] \big|\,   \frac{\d t}{t} >\delta \bigg)  \\
&\le \frac{1}{\delta^2} \limsup_{n\to\infty}\Big( \int_{M/2}^{\infty} \sqrt{ \tau_{k,n}^{-1} \V \big( \nu_{k,n}(t) \big)} \, \frac{\d t }{t} \Big)^2.
\end{align*}
By following calculations for the covariance asymptotics in the proof of Lemma \ref{l:exp.var.asym} $(ii)$ and using Potter's bounds (for the application of the dominated convergence theorem),  we can derive  that for any $0 < \vep < \alpha -d $,
\begin{align*}
\tau_{k,n}^{-1} \V \big( \nu_{k,n}(t) \big) & \le C^*\sum_{\ell = 1}^k \int_{t}^{\infty} \rho^{d - 1 - (2k - \ell) (\alpha - \vep)} \d \rho \le C^* \sum_{\ell = 1}^k t^{d - (2k - \ell) (\alpha - \vep)}, 
\end{align*}
which in turn implies that 
$$
\limsup_{n\to\infty} \int_{M/2}^{\infty} \sqrt{ \tau_{k,n}^{-1} \V \big( \nu_{k,n}(t) \big) }\,  \frac{\d t}{t}  \le C^* \sum_{\ell = 1}^k \lp \frac{M}{2} \rp^{\frac{d}{2} - \frac{(2k - \ell) (\alpha - \vep)}{2}}.
$$
The last term converges to $0$ as $M\to\infty$, so the proof is completed. 
\end{proof}

\subsection{Proof of Theorem \ref{t:adjust.centering}}

Throughout this section, we assume Conditions (C1)--(C3) in Section \ref{sec:non.random.centering}. Recall that we have already established in Lemma \ref{l:exp.var.asym} $(i)$ that for every $t>0$, 
\begin{equation}  \label{e:pointwise.conv.t}
m^{-k}\E\big[ \nu_{k,n} (t)\big] \to t^{d-\alpha k}, \ \  \text{ as } n \to \infty.
\end{equation}
The next lemma provides a result stronger than \eqref{e:pointwise.conv.t} under Conditions (C1)--(C3). 

\begin{lemma} \label{l:stronger.mean.conv}
$(i)$ For every $k \ge1$ and  $\delta > 0$, 
\begin{equation}  \label{e:uniform.conv.mean}
\sup_{t \in [\delta, \infty]} m^k \tau_{k,n}^{-1/2} \left| m^{-k} \E \lb \nu_{k,n}(t) \rb -  t^{d - \alpha k} \right| \to 0, \ \  \text{ as } n \to \infty.
\end{equation}
$(ii)$ For every $k\ge1$, 
$$
m^k \tau_{k,n}^{-1/2} \int_1^\infty \big( m^{-k} \E[\nu_{k,n}(t)]-t^{d-\alpha k} \big) \, \frac{\d t}{t} \to 0, \ \ \text{as } n\to\infty. 
$$
\end{lemma}

\begin{proof}[Proof of Lemma \ref{l:stronger.mean.conv}]
For the proof of $(i)$, observe that by the definition of $\tau_{k,n}$ in \eqref{e:def.tau_kn}, we have $m^k \tau_{k,n}^{-1/2} \le C^*m^{k/2}$. It then follows from Conditions (C1) and (C2), together with Potter's bounds applied to the ratios of densities as in \eqref{e:Potter1} and \eqref{e:Potter2}, that \eqref{e:uniform.conv.mean} is bounded above, up to the scale, by 
\begin{align}   
&\sup_{t\in [\delta, \infty]} m^{k/2} \int_{\rho=t}^\infty \int_{\theta\in \mathbb{S}^{d-1}} \int_{\bz\in (\R^d)^{k-1}} h_k(0,\bz) \label{e:after.C1.C2}  \\
&\qquad \qquad\qquad \qquad\qquad  \times \Big( 1-\1\big\{ \ms{min}(\rho\theta + \Rk^{-1}\bz)\ge t \big\} \Big)\, \rho^{d-1-\alpha k} \d \bz \d \sigma(\theta) \d \rho. \notag
\end{align}
For every $\rho\ge t$, $\theta \in \mathbb{S}^{d-1}$, and $\bz\in (\R^d)^{k-1}$ with $h_k(0,\bz)=1$, 
\begin{align*}
1-\1\big\{ \ms{min}(\rho\theta + \Rk^{-1}\bz)\ge t \big\} &\le \sum_{i=1}^{k-1} \1 \big\{\, |\rho\theta +\Rk^{-1}z_i| <t \big\}  \\
&\le (k-1) \1 \big\{ \rho \le t + \Rk^{-1}L \big\}. 
\end{align*}
Hence, one can bound \eqref{e:after.C1.C2} by 
\begin{align*}
&(k-1)s_{d-1} D_{k,k} \sup_{t\in [\delta, \infty]} m^{k/2} \int_t^{t+\Rk^{-1}L} \rho^{d-1-\alpha k} \d \rho \\
&\le \frac{(k-1)s_{d-1}D_{k,k}}{\alpha d-k}\, \delta^{d-\alpha k}  m^{k/2} \Big[ 1-\Big( 1+\frac{L}{\Rk \delta} \Big)^{d-\alpha k} \Big] \\
&\le C^* m^{k/2} \Rk^{-1} \to 0. 
\end{align*}
The last convergence is assured by Condition (C3). 

The proof of $(ii)$ is similar to the above, so we will omit it here. 
\end{proof}

Combining Lemma \ref{l:stronger.mean.conv} $(i)$ and Proposition \ref{p:CLT.functional}, we have, as $n\to\infty$, 
\begin{equation} \label{e:non-random.center.funcl.CLT}
\lp \lp m^k\tau_{k,n}^{-1/2} \big( m^{-k}\nu_{k,n}(t) - t^{d-\alpha k} \big), \, t \in (0,\infty] \rp\rp_{k=1}^K \Rightarrow \lp \lp W_k(t), \, t \in (0,\infty] \rp \rp_{k=1}^K, 
\end{equation}
in the space $\big(\mathcal D(0,\infty]\big)^K$. The remainder of the argument follows \cite[Chapter 9]{Resnick:2007}, where the key component is Vervaat's lemma (\cite{vervaat:1972}). For this purpose, our first step is to establish the functional-level asymptotic normality for the processes 
\begin{align*}
A_{k,n}(s) &:= m^k \tau_{k,n}^{-1/2} \big(m^{-k}\nu_{k,n}(s^{(d-\alpha k)^{-1}}) -s \big), \ \ \ s \in [0,\infty), \\
B_{k,n}(s) &:= m^k \tau_{k,n}^{-1/2} \Big( \Big(\frac{U_k(\lceil m^k s\rceil)}{\Rk}\Big)^{d-\alpha k} -s \Big),   \ \ \ s \in [0,\infty), 
\end{align*}
where $\lceil a \rceil$ denotes the smallest integer exceeding $a$. 

\begin{lemma}  \label{l:CLT.vervaat}
For every $K\ge1$, as $n\to\infty$, 
\begin{align*}
&\Big( \big( A_{k,n}(s), \, s\in [0,\infty) \big), \, \big( B_{k,n}(s), \, s\in [0,\infty) \big) \Big)_{k=1}^K \\
&\quad \Rightarrow \Big( \big( W_{k}(s^{(d-\alpha k)^{-1}}), \, s\in [0,\infty) \big), \, \big( -W_{k}(s^{(d-\alpha k)^{-1}}), \, s\in [0,\infty) \big) \Big)_{k=1}^K, 
\end{align*}
in the space $\big(\mathcal D[0,\infty)\big)^{2K}$. 
\end{lemma}

\begin{proof}[Proof of Lemma \ref{l:CLT.vervaat}]
By setting $s=t^{d-\alpha k}$, \eqref{e:non-random.center.funcl.CLT} is equivalent to 
$$
\Big( \big( A_{k,n}(s), \, s\in [0,\infty) \big)\Big)_{k=1}^K \Rightarrow \Big(\big( W_k (s^{(d-\alpha k)^{-1}}), \, s \in [0,\infty)\big) \Big)_{k=1}^K, \ \ \text{ in } \big(\mathcal D[0,\infty)\big)^K. 
$$
For any $t>0$, the (left-continuous) inverse of $m^{-k} \nu_{k,n} (\, \cdot^{(d-\alpha k)^{-1}})$ satisfies the following: 
\begin{align*}
\big( m^{-k} \nu_{k,n} (\, \cdot^{(d-\alpha k)^{-1}}) \big)^{\leftarrow}(s) := \inf \big\{  r: \nu_{k,n}(r^{(d-\alpha k)^{-1}}) \ge m^k s\big\} = \Big( \frac{U_k(\lceil m^k s \rceil)}{\Rk} \Big)^{d-\alpha k}. 
\end{align*}
Now, the desired result follows from   Vervaat's lemma in the Appendix.   
\end{proof}

We will present an additional lemma below. 

\begin{lemma}\label{l:Hill.CLT.augmented}
It holds that as $n\to\infty$, 
\begin{align*}
&\bigg( m^k\tau_{k,n}^{-1/2} \Big( H_{k,m,n} -m^{-k}\int_1^\infty \E \big[\nu_{k,n}(s)\big] \Big|_{s=\frac{\consisest t}{\Rk}} \frac{\d t}{t}\Big), \, \\
& \qquad \qquad \qquad \qquad \qquad \qquad  m^k\tau_{k,n}^{-1/2} \Big(  \Big( \frac{\consisest}{\Rk} \Big)^{d-\alpha k}-1 \Big)  \bigg)_{k=1}^K \\
&\quad \Rightarrow \Big( \int_1^\infty W_k(t)\frac{\d t}{t}, \, -W_k(1) \Big)_{k=1}^K, \ \ \text{ in } \R^{2K}. 
\end{align*}
\end{lemma}

\begin{proof}[Proof of Lemma \ref{l:Hill.CLT.augmented}]
Reverting the time parameter in Lemma \ref{l:CLT.vervaat} back to the original  $t=s^{(d-\alpha k)^{-1}}$, and appealing to Lemma \ref{l:stronger.mean.conv} $(i)$ again, as well as the augmentation via Proposition \ref{p:consistency.tail.measure} $(ii)$, it holds that 
\begin{align*}
&\bigg( \Big( \tau_{k,n}^{-1/2} \big( \nu_{k,n}(t) -\E[\nu_{k,n}(t)] \big), \, t\in (0,\infty] \Big), \, m^k \tau_{k,n}^{-1/2} \Big( \Big( \frac{\consisest}{\Rk} \Big)^{d-\alpha k}-1 \Big), \, \\
& \qquad \qquad \qquad \qquad \qquad \qquad \qquad \qquad \qquad \qquad \qquad \qquad \qquad \qquad \quad \frac{\consisest}{\Rk}  \bigg)_{k=1}^K \\
&\quad \Rightarrow \Big( \big( W_k(t), \, t\in (0,\infty] \big), \, -W_k(1), \, 1 \Big)_{k=1}^K, \ \ \text{in } \big( \mathcal D(0,\infty] \times \R \times (0,\infty) \big)^K. 
\end{align*}
The assertion can now be obtained via the composition map $J: \mathcal{D}(0, \infty] \times \R \times (0, \infty) \to \mathcal{D}(0, \infty] \times \R$, defined by $J(x(\cdot), z, r) = (x(r \cdot), z)$ (this map is continuous on the support of $W_k$). This is followed by the integration map $x(\cdot) \mapsto \int_1^\infty x(t) t^{-1} \, \d t$. Since this integration map is not continuous, we first apply its truncated and continuous version, $x(\cdot) \mapsto \int_1^M x(t) t^{-1} \, \d t$ with $M \in (1, \infty)$. Finally, an approximation argument is used, following \cite[Theorem 3.5]{Resnick:2007}, as in the proof of \eqref{e:second.converging.Hill.asym.normality}.
\end{proof}
\medskip

\begin{proof}[Proof of Theorem \ref{t:adjust.centering}]
The first component in Lemma \ref{l:Hill.CLT.augmented} can be divided into three terms: for any $k=1,\dots,K$, 
\begin{align}
&m^k\tau_{k,n}^{-1/2} \Big( H_{k,m,n} -m^{-k}\int_1^\infty \E \big[\nu_{k,n}(s)\big] \Big|_{s=\frac{\consisest t}{\Rk}} \frac{\d t}{t}\Big) \label{e:decomp.into.3} \\
&\quad = m^k\tau_{k,n}^{-1/2}  \big( H_{k,m,n} - (\alpha k-d)^{-1} \big) \notag \\
& \qquad + \tau_{k,n}^{-1/2} \int_1^\infty \Big( \E[\nu_{k,n}(t)] - \E[ \nu_{k,n}(s)]\Big|_{s=\frac{\consisest t}{\Rk}} \Big) \frac{\d t}{t} \notag \\
&\qquad -m^k \tau_{k,n}^{-1/2} \int_1^\infty \big( m^{-k} \E[\nu_{k,n}(t)]-t^{d-\alpha k} \big) \frac{\d t}{t} =: A_{1,n} + A_{2,n} - A_{3,n}.\notag 
\end{align}
Among these terms, Lemma \ref{l:stronger.mean.conv} $(ii)$ ensures that $A_{3,n}\to0$ as $n\to\infty$. 

We next claim that
\begin{equation} \label{e:second.comp.ratio} 
A_{2,n}= \frac{m^k \tau_{k,n}^{-1/2}}{\alpha k - d} \Big( 1 - \Big( \frac{\consisest}{\Rk} \Big)^{d - \alpha k} \Big) + o_p(1).
\end{equation}
Assume temporarily that \eqref{e:second.comp.ratio} holds. It then follows from Lemma \ref{l:Hill.CLT.augmented} and \eqref{e:decomp.into.3}, as well as \eqref{e:second.comp.ratio}, that as $n \to \infty$,
\begin{equation} \label{e:W_k.functional}
\Big( m^k \tau_{k,n}^{-1/2} \big( H_{k,m,n} - (\alpha k - d)^{-1} \big) \Big)_{k=1}^K \Rightarrow \bigg( \int_1^\infty W_k(t) \frac{\d t}{t} - \frac{W_k(1)}{\alpha k - d} \bigg)_{k=1}^K.
\end{equation}
It is straightforward to verify that the law of the right-hand side in \eqref{e:W_k.functional} is equivalent  to the law of $(Z_k)_{k=1}^K$ in Theorem \ref{t:adjust.centering}, which can be done by showing that the covariances of both random variables coincide.

It now remains to prove \eqref{e:second.comp.ratio}. First, by the change of variables as in \eqref{e:first.change.of.variables}--\eqref{e:final.form.expectation}, $A_{2,n}$ can be denoted as 
\begin{align}
& \frac{m^k \tau_{k,n}^{-1/2}}{k!} \lp \frac{n}{m} \rp^k \Rk^d f \lp \Rk \rp^k \label{e:diff.two.exp}\\
& \qquad \qquad \qquad \qquad \qquad \times \int_{t = 1}^{\infty} \int_{\rho = 0}^{\infty} \int_{\theta \in \mathbb{S}^{d-1}} \int_{\bz \in (\R^d)^{k-1}} h_k(0, \bz) \rho^{d-1}  \notag \\
& \times \bigg[ \1 \{ \ms{min}(\rho \theta, \rho \theta + \Rk^{-1} \bz) \ge t \} \notag \\
& \qquad \qquad \qquad   - \1 \bigg\{ \ms{min} (\rho \theta, \rho \theta + \Rk^{-1} \bz) \ge \frac{\consisest t}{\Rk}  \bigg\} \bigg] \notag \\
& \times \frac{f \lp \Rk \rho \rp}{f \lp \Rk \rp} \prod_{i=1}^{k-1} \frac{f \lp \Rk | \rho \theta + \Rk^{-1} z_i|  \rp}{f \lp \Rk \rp} \d \bz \d \sigma(\theta) \d \rho \frac{\d t}{t}.  \notag 
\end{align}
Note that
\begin{align}
& \bigg| \1 \{ \ms{min}(\rho \theta, \rho \theta + \Rk^{-1} \bz) \ge t \}\label{e:bound.diff.indicators} \\
& \qquad  \qquad  - \1\bigg\{ \ms{min}(\rho \theta, \rho \theta +\Rk^{-1} \bz) \ge \frac{\consisest t}{\Rk}  \bigg\} \bigg| \notag \\
& \leq \1\{ \rho \ge t \} + \1 \left\{ \rho \geq \frac{\consisest t}{\Rk} \right\}. \notag 
\end{align}
Because of Conditions (C1) and (C2), Proposition \ref{p:consistency.tail.measure} $(ii)$, the bound in \eqref{e:bound.diff.indicators}, and Potter's bounds applied to the ratios of densities, \eqref{e:diff.two.exp} is asymptotically equal to 
\begin{align}
&\frac{\alpha k -d}{k! (C_k)^k}\, m^k \tau_{k,n}^{-1/2} \int_{t = 1}^{\infty} \int_{\rho = 0}^{\infty} \int_{\theta \in \mathbb{S}^{d-1}} \int_{\bz \in (\R^d)^{k-1}} h_k(0, \bz) \rho^{d-1-\alpha k} \label{e:difference.only.indicators} \\
& \times \bigg[ \1 \{ \ms{min}(\rho \theta, \rho \theta + \Rk^{-1} \bz) \ge t \} \notag \\
&\qquad  \quad - \1 \left\{ \ms{min} (\rho \theta, \rho \theta + \Rk^{-1} \bz) \ge \frac{\consisest t}{\Rk}  \right\} \bigg] \d \bz \d \sigma(\theta) \d \rho \frac{\d t}{t} + o_p(1). \notag
\end{align}
Subsequently, we need to justify that the two indicator functions in \eqref{e:difference.only.indicators}  can be asymptotically replaced with $\1\{ \rho\ge t\}$ and $\big\{ \rho \ge \Rk^{-1} \consisest t  \big\}$, respectively. To verify this claim for the latter indicator (the former can be treated analogously), observe that
\begin{align*}
0 &\le \1 \Big\{ \rho \ge \frac{\consisest t}{\Rk}  \Big\} - \1 \Big\{ \ms{min}(\rho\theta, \rho \theta + \Rk^{-1} \bz )\ge \frac{\consisest t}{\Rk} \Big\}  \\
&= \1 \Big\{ \rho \ge \frac{\consisest t}{\Rk}  \Big\}\, \prod_{i=1}^{k-1} \1 \Big\{ \, \Big| \rho \theta +\frac{z_i}{\Rk}  \Big| < \frac{\consisest t}{\Rk}  \Big\}  \\
&\le (k-1) \1 \Big\{ \frac{\consisest t}{\Rk} \le \rho < \frac{L}{\Rk} +  \frac{\consisest t}{\Rk} \Big\}. 
\end{align*}
Then, 
\begin{align*}
&m^k \tau_{k,n}^{-1/2} \int_{t = 1}^{\infty} \int_{\rho = 0}^{\infty} \int_{\theta \in \mathbb{S}^{d-1}} \int_{\bz \in (\R^d)^{k-1}} h_k(0, \bz) \rho^{d-1-\alpha k} \\
&\qquad \qquad \qquad \times  \1 \Big\{ \frac{\consisest t}{\Rk} \le \rho < \frac{L}{\Rk} +  \frac{\consisest t}{\Rk} \Big\} \d \bz \d \sigma(\theta) \d \rho \frac{\d t}{t} \\
&\le C^* m^k \tau_{k,n}^{-1/2} \Big( \frac{\consisest}{\Rk} \Big)^{d-\alpha k} \bigg(  1-\Big( 1+\frac{L}{\consisest} \Big)^{d-\alpha k}\bigg). 
\end{align*}
By Proposition \ref{p:consistency.tail.measure} $(ii)$ and Condition (C3), as well as $m^k \tau_{k,n}^{-1/2}\le C^* m^{k/2}$, the last term converges to $0$ in probability. 

Summarizing these results, the main term in \eqref{e:difference.only.indicators} is asymptotically equivalent  to 
\begin{align*}
&\frac{\alpha k -d}{k! (C_k)^k}\, m^k \tau_{k,n}^{-1/2} \int_{t = 1}^{\infty} \int_{\rho = 0}^{\infty} \int_{\theta \in \mathbb{S}^{d-1}} \int_{\bz \in (\R^d)^{k-1}} h_k(0, \bz) \rho^{d-1-\alpha k}  \\
&\qquad  \qquad \qquad \qquad \qquad \times \bigg[ \1 \{ \rho \ge t \} - \1 \Big\{ \rho \ge \frac{\consisest t}{\Rk}  \Big\} \bigg] \d \bz \d \sigma(\theta) \d \rho \frac{\d t}{t} \\
&= \frac{m^k \tau_{k,n}}{\alpha k-d}\,\bigg( 1-\Big( \frac{\consisest}{\Rk} \bigg)^{d-\alpha k} \Big). 
\end{align*}
Now, \eqref{e:second.comp.ratio} has been verified. 
\end{proof}

\begin{acks}[Acknowledgments]
The authors would like to thank the anonymous referees for their insightful comments and suggestions that improved the quality of this paper. Particularly, their invaluable comments on more recent techniques for normal approximation for Poisson U-functionals and the properties of the space of signed Radon measures, enriched the authors' discussion. Most of TK's work was done during his Ph.D. in the Department of Statistics at Purdue University.
\end{acks}

\begin{funding}
TO's research was partially supported by the AFOSR grant FA9550-22-1-0238 at Purdue University.
\end{funding}

\bibliographystyle{imsart-number}
\bibliography{Hill.bib}

\begin{thebibliography}{29}

\bibitem{Adler;Bobrowski;Weinberger:2014}
\begin{barticle}[author]
\bauthor{\bsnm{Adler},~\bfnm{R.~J.}\binits{R.~J.}}, \bauthor{\bsnm{Bobrowski},~\bfnm{O.}\binits{O.}} \AND \bauthor{\bsnm{Weinberger},~\bfnm{S.}\binits{S.}}
(\byear{2014}).
\btitle{Crackle: {T}he homology of noise}.
\bjournal{Discrete and Computational Geometry}
\bvolume{52}
\bpages{680-704}.
\end{barticle}
\endbibitem

\bibitem{beirlant:alves:gomes:2016}
\begin{barticle}[author]
\bauthor{\bsnm{Beirlant},~\bfnm{J.}\binits{J.}}, \bauthor{\bsnm{Alves},~\bfnm{I.~F.}\binits{I.~F.}} \AND \bauthor{\bsnm{Gomes},~\bfnm{I.}\binits{I.}}
(\byear{2016}).
\btitle{Tail fitting for truncated and non-truncated {P}areto-type distributions}.
\bjournal{Extremes}
\bvolume{19}
\bpages{429-462}.
\end{barticle}
\endbibitem

\bibitem{Berthet;Einmahl:2022}
\begin{barticle}[author]
\bauthor{\bsnm{Berthet},~\bfnm{P.}\binits{P.}} \AND \bauthor{\bsnm{Einmahl},~\bfnm{J.~H.~J.}\binits{J.~H.~J.}}
(\byear{2022}).
\btitle{Cube root weak convergence of empirical estimators of a density level set}.
\bjournal{Annals of Statistics}
\bvolume{50}
\bpages{1423-1446}.
\end{barticle}
\endbibitem

\bibitem{Billingsley:1999}
\begin{bbook}[author]
\bauthor{\bsnm{Billingsley},~\bfnm{P.}\binits{P.}}
(\byear{1999}).
\btitle{Convergence of Probability Measures},
\bedition{second} ed.
\bpublisher{Wiley}.
\end{bbook}
\endbibitem

\bibitem{Bingham;Goldie;Teugels:1987}
\begin{bbook}[author]
\bauthor{\bsnm{Bingham},~\bfnm{N.~H.}\binits{N.~H.}}, \bauthor{\bsnm{Goldie},~\bfnm{C.~M.}\binits{C.~M.}} \AND \bauthor{\bsnm{Teugels},~\bfnm{J.~L.}\binits{J.~L.}}
(\byear{1987}).
\btitle{Regular Variation}.
\bpublisher{Cambridge University Press}.
\end{bbook}
\endbibitem

\bibitem{burroughs:tebbens:2001}
\begin{barticle}[author]
\bauthor{\bsnm{Burroughs},~\bfnm{S.~M.}\binits{S.~M.}} \AND \bauthor{\bsnm{Tebbens},~\bfnm{S.~F.}\binits{S.~F.}}
(\byear{2001}).
\btitle{Upper-truncated Power Laws in Natural Systems}.
\bjournal{Pure and Applied Geophysics}
\bvolume{158}
\bpages{741-757}.
\end{barticle}
\endbibitem

\bibitem{burroughs:tebbens:2002}
\begin{barticle}[author]
\bauthor{\bsnm{Burroughs},~\bfnm{S.~M.}\binits{S.~M.}} \AND \bauthor{\bsnm{Tebbens},~\bfnm{S.~F.}\binits{S.~F.}}
(\byear{2002}).
\btitle{The upper-truncated power law applied to earthquake cumulative frequency-magnitude distributions: evidence for a time-independent scaling parameter}.
\bjournal{Bulletin of the Seismological Society of America}
\bvolume{92}
\bpages{2983–2993}.
\end{barticle}
\endbibitem

\bibitem{chakrabarty:samorodnitsky:2012}
\begin{barticle}[author]
\bauthor{\bsnm{Chakrabarty},~\bfnm{A.}\binits{A.}} \AND \bauthor{\bsnm{Samorodnitsky},~\bfnm{G.}\binits{G.}}
(\byear{2012}).
\btitle{Understanding heavy tails in a bounded world or, is a truncated heavy tail heavy or not?}
\bjournal{Stochastic Models}
\bvolume{28}
\bpages{109-143}.
\end{barticle}
\endbibitem

\bibitem{Barrio;Deheuvels;Geer:2007}
\begin{bbook}[author]
\bauthor{\bparticle{del} \bsnm{Barrio},~\bfnm{Eustasio}\binits{E.}}, \bauthor{\bsnm{Deheuvels},~\bfnm{Paul}\binits{P.}} \AND \bauthor{\bparticle{van~de} \bsnm{Geer},~\bfnm{Sara}\binits{S.}}
(\byear{2007}).
\btitle{Lectures on Empirical Processes}.
\bpublisher{European Mathematical Society}.
\end{bbook}
\endbibitem

\bibitem{embrechts:kluppelberg:mikosch:1997}
\begin{bbook}[author]
\bauthor{\bsnm{Embrechts},~\bfnm{P.}\binits{P.}}, \bauthor{\bsnm{Kl\"uppelberg},~\bfnm{C.}\binits{C.}} \AND \bauthor{\bsnm{Mikosch},~\bfnm{T.}\binits{T.}}
(\byear{1997}).
\btitle{Modelling Extremal Events: for Insurance and Finance}.
\bpublisher{Springer}, \baddress{New York}.
\end{bbook}
\endbibitem

\bibitem{Geluk;deHaan;Resnick;Starica:1997}
\begin{barticle}[author]
\bauthor{\bsnm{Geluk},~\bfnm{J.}\binits{J.}}, \bauthor{\bparticle{de} \bsnm{Haan},~\bfnm{L.}\binits{L.}}, \bauthor{\bsnm{Resnick},~\bfnm{S.~I.}\binits{S.~I.}} \AND \bauthor{\bsnm{St\v{a}ric\v{a}},~\bfnm{C.}\binits{C.}}
(\byear{1997}).
\btitle{Second-order regular variation, convolution and the central limit theorem}.
\bjournal{Stochastic Processes and their Applications}
\bvolume{69}
\bpages{139-159}.
\end{barticle}
\endbibitem

\bibitem{Hill:1975}
\begin{barticle}[author]
\bauthor{\bsnm{Hill},~\bfnm{B.~M.}\binits{B.~M.}}
(\byear{1975}).
\btitle{A simple general approach to inference about the tail of a distribution}.
\bjournal{The Annals of Statistics}
\bvolume{3}
\bpages{1163-1174}.
\end{barticle}
\endbibitem

\bibitem{Horowitz:1986}
\begin{barticle}[author]
\bauthor{\bsnm{Horowitz},~\bfnm{J.}\binits{J.}}
(\byear{1986}).
\btitle{Gaussian random measures}.
\bjournal{Stochastic Processes and their Applications}
\bvolume{22}
\bpages{129-133}.
\end{barticle}
\endbibitem

\bibitem{Last;Penrose:2017}
\begin{bbook}[author]
\bauthor{\bsnm{Last},~\bfnm{G.}\binits{G.}} \AND \bauthor{\bsnm{Penrose},~\bfnm{M.}\binits{M.}}
(\byear{2017}).
\btitle{Lectures on the {P}oisson Process},
\bedition{first} ed.
\bpublisher{Cambridge University Press}.
\end{bbook}
\endbibitem

\bibitem{Owada:2017}
\begin{barticle}[author]
\bauthor{\bsnm{Owada},~\bfnm{T.}\binits{T.}}
(\byear{2017}).
\btitle{Functional central limit theorem for subgraph counting processes}.
\bjournal{Electronic Journal of Probability}
\bvolume{22}
\bpages{1-38}.
\end{barticle}
\endbibitem

\bibitem{Owada:2018}
\begin{barticle}[author]
\bauthor{\bsnm{Owada},~\bfnm{T.}\binits{T.}}
(\byear{2018}).
\btitle{Limit theorems for {B}etti numbers of extreme sample clouds with application to persistence barcodes}.
\bjournal{The Annals of Applied Probability}
\bvolume{28}
\bpages{2814-2854}.
\end{barticle}
\endbibitem

\bibitem{Owada;Adler:2017}
\begin{barticle}[author]
\bauthor{\bsnm{Owada},~\bfnm{T.}\binits{T.}} \AND \bauthor{\bsnm{Adler},~\bfnm{R.~J.}\binits{R.~J.}}
(\byear{2017}).
\btitle{Limit theorems for point processes under geometric constraints (and topological crackle)}.
\bjournal{The Annals of Probability}
\bvolume{45}
\bpages{2004-2055}.
\end{barticle}
\endbibitem

\bibitem{owada:bobrowski:2020}
\begin{barticle}[author]
\bauthor{\bsnm{Owada},~\bfnm{T.}\binits{T.}} \AND \bauthor{\bsnm{Bobrowski},~\bfnm{O.}\binits{O.}}
(\byear{2020}).
\btitle{Convergence of persistence diagrams for topological crackle}.
\bjournal{Bernoulli}
\bvolume{26}
\bpages{2275-2310}.
\end{barticle}
\endbibitem

\bibitem{Penrose:2003}
\begin{bbook}[author]
\bauthor{\bsnm{Penrose},~\bfnm{M.}\binits{M.}}
(\byear{2003}).
\btitle{Random Geometric Graphs}.
\bseries{Oxford Studies in Probability}.
\bpublisher{Oxford University Press}.
\end{bbook}
\endbibitem

\bibitem{Reitzner;Schulte:2013}
\begin{barticle}[author]
\bauthor{\bsnm{Reitzner},~\bfnm{Matthias}\binits{M.}} \AND \bauthor{\bsnm{Schulte},~\bfnm{Matthias}\binits{M.}}
(\byear{2013}).
\btitle{Central limit theorems for {U}-statistics of {P}oisson point processes}.
\bjournal{The Annals of Probability}
\bvolume{41}
\bpages{3879-3909}.
\end{barticle}
\endbibitem

\bibitem{resnick:1987}
\begin{bbook}[author]
\bauthor{\bsnm{Resnick},~\bfnm{S.~I.}\binits{S.~I.}}
(\byear{1987}).
\btitle{Extreme Values, Regular Variation and Point Processes}.
\bpublisher{Springer New York}.
\end{bbook}
\endbibitem

\bibitem{Resnick:2007}
\begin{bbook}[author]
\bauthor{\bsnm{Resnick},~\bfnm{S.~I.}\binits{S.~I.}}
(\byear{2007}).
\btitle{Heavy-Tail Phenomena}.
\bpublisher{Springer-Verlag New York}.
\end{bbook}
\endbibitem

\bibitem{dehaan:ferreira:2006}
\begin{bbook}[author]
\bauthor{\bsnm{{\rm de Haan}},~\bfnm{L.}\binits{L.}} \AND \bauthor{\bsnm{Ferreira},~\bfnm{A.}\binits{A.}}
(\byear{2006}).
\btitle{Extreme Value Theory: An Introduction}.
\bpublisher{Springer}, \baddress{New York}.
\end{bbook}
\endbibitem

\bibitem{Schulte:2016}
\begin{barticle}[author]
\bauthor{\bsnm{Schulte},~\bfnm{Matthias}\binits{M.}}
(\byear{2016}).
\btitle{Normal approximation of {P}oisson functionals in {K}olmogorov distance}.
\bjournal{Journal of Theoretical Probability}
\bvolume{29}
\bpages{96-117}.
\end{barticle}
\endbibitem

\bibitem{Thomas:2023}
\begin{barticle}[author]
\bauthor{\bsnm{Thomas},~\bfnm{A.~M.}\binits{A.~M.}}
(\byear{2023}).
\btitle{Central limit theorems and asymptotic independence for local {U}-statistics on diverging halfspaces}.
\bjournal{Bernoulli}
\bvolume{29}
\bpages{3280-3306}.
\end{barticle}
\endbibitem

\bibitem{vervaat:1972}
\begin{barticle}[author]
\bauthor{\bsnm{Vervaat},~\bfnm{W.}\binits{W.}}
(\byear{1972}).
\btitle{Functional central limit theorems for processes with positive drift and their inverses}.
\bjournal{Zeitschrift f\"ur Wahrscheinlichkeitstheorie und Verwandte Gebiete}
\bvolume{23}
\bpages{245–253}.
\end{barticle}
\endbibitem

\bibitem{Wei;Owada:2022}
\begin{barticle}[author]
\bauthor{\bsnm{Wei},~\bfnm{Z.}\binits{Z.}} \AND \bauthor{\bsnm{Owada},~\bfnm{T.}\binits{T.}}
(\byear{2022}).
\btitle{Functional strong law of large numbers for {B}etti numbers in the tail}.
\bjournal{Extremes}
\bvolume{25}
\bpages{653-693}.
\end{barticle}
\endbibitem

\bibitem{Xu;Davis;Samorodnitsky:2022}
\begin{barticle}[author]
\bauthor{\bsnm{Xu},~\bfnm{H.}\binits{H.}}, \bauthor{\bsnm{Davis},~\bfnm{R.}\binits{R.}} \AND \bauthor{\bsnm{Samorodnitsky},~\bfnm{G.}\binits{G.}}
(\byear{2022}).
\btitle{Handling missing extremes in tail estimation}.
\bjournal{Extremes}
\bvolume{25}
\bpages{199-227}.
\end{barticle}
\endbibitem

\bibitem{Zou;Davis;Samorodnitsky:2020}
\begin{barticle}[author]
\bauthor{\bsnm{Zou},~\bfnm{J.}\binits{J.}}, \bauthor{\bsnm{Davis},~\bfnm{R.~A.}\binits{R.~A.}} \AND \bauthor{\bsnm{Samorodnitsky},~\bfnm{G.}\binits{G.}}
(\byear{2020}).
\btitle{Extreme value analysis without the largest values: what can be done?}
\bjournal{Probability in the Engineering and Information Sciences}
\bvolume{34}
\bpages{200-220}.
\end{barticle}
\endbibitem

\end{thebibliography}


\appendix

\section{Limit value of $nf\big( R_k(C_kn/m) \big)$}

\begin{lemma}  \label{l:no.oscillation}
Under the assumptions of Section \ref{sec:formal.def}, for each $k\ge1$  there exists an increasing function $R_k:[0,\infty)\to [0,\infty)$  satisfying \eqref{e:def.radius.function.R}. Moreover, the limit of $nf\big( R_k(C_kn/m) \big)$ exists in $[0,\infty]$  without oscillation, where $C_k$ is the constant defined in \eqref{e:def.Ck}. More precisely, we have the following results. 
\vspace{2pt}

\noindent $(i)$ If $\beta<d/(\alpha k)$ or $\beta=d/(\alpha k)$ and $L(t)\to\infty$ as $t\to\infty$, then $nf\big( R_k(C_kn/m) \big) \to 0$ as $n\to\infty$. \\
$(ii)$  If $\beta>d/(\alpha k)$ or $\beta=d/(\alpha k)$ and $L(t)\to0$ as $t\to\infty$, then $nf\big( R_k(C_kn/m) \big) \to \infty$ as $n\to\infty$. \\
$(iii)$ If $\beta=d/(\alpha k)$ and $L(t)$ converges to a finite, positive constant as $t\to\infty$, then $nf\big( R_k(C_kn/m) \big) \to \xi\in (0,\infty)$ as $n\to\infty$. 
\end{lemma}

\begin{proof}
Given $k\ge1$, define $q(r):= r^{-d/k}f(r)^{-1}$. Then $q$ is regularly  varying with index $\alpha-d/k$. Moreover,  $q$ is eventually increasing, since by \eqref{e:smoothness.L},
$$
\frac{\dif }{\dif  r}\log q(r)= r^{-1}\Big( \alpha-\frac{d}{k}-\frac{rL'(r)}{L(r)} \Big)>0. 
$$
We now define $R_k(t)$ as the (left-continuous) inverse of $q$:
$$
R_k(t):=q^\leftarrow \Big( \frac{t}{(\alpha k-d)^{1/k}} \Big) = \inf \big\{  r: q(r) \geq t(\alpha k-d)^{-1/k}\big\}, \ \ \ t>0. 
$$
It then follows from Proposition 2.6 (v) in \cite{Resnick:2007} that $R_k(t)$ is regularly varying with index  $k/(\alpha k-d)$ and diverges to infinity. Moreover, it satisfies
\begin{equation}  \label{e:original.inverse}
q\big( R_k(t) \big) = \frac{t}{(\alpha k-d)^{1/k}}, \ \ \text{ for large } t. 
\end{equation}
By the definition of $q$, we have 
$$
t^k R_k(t)^d f\big( R_k(t) \big)^k  = \Big( \frac{t}{q\big( R_k(t) \big)} \Big)^k. 
$$
By \eqref{e:original.inverse}, the right-hand side equals $\alpha k - d$ for large $t$, and hence, \eqref{e:def.radius.function.R} follows. 

Subsequently, we discuss the asymptotics  of $n f\big( R_k(C_k n/m) \big)$. Using \eqref{e:def.radius.function.R}, we find that 
\begin{equation}  \label{e:nf.1st.step}
nf\big( R_k(C_k n/m) \big) \sim \frac{(\alpha k-d)^{1/k}}{C_k}\cdot \frac{m}{R_k(C_kn/m)^{d/k}}, \ \ \ n\to\infty. 
\end{equation}
By the representation $q(r)=r^{(\alpha k-d)/k}L(r)^{-1}$, we have 
$$
q\big( R_k(C_kn/m) \big) = R_k(C_kn/m)^{(\alpha k-d)/k} L\big( R_k(C_kn/m) \big)^{-1}. 
$$
Furthermore, \eqref{e:original.inverse} implies 
$$
q\big( R_k(C_kn/m) \big) = \frac{C_k}{(\alpha k-d)^{1/k}} \cdot \frac{n}{m}, 
$$
for sufficiently large $n$ and thus,
\begin{equation}  \label{e:Rk.expression}
R_k(C_kn/m)^{-d/k} = \bigg\{ \frac{C_k}{(\alpha k-d)^{1/k}} \cdot \frac{n}{m} L\big( R_k(C_kn/m) \big)\bigg\}^{-d/(\alpha k-d)}, 
\end{equation}
for large $n$. 
Substituting \eqref{e:Rk.expression} and \eqref{e:RV.m(n)} into \eqref{e:nf.1st.step} yields 
\begin{equation}   \label{e:limit.nf}
nf\big( R_k(C_k n/m) \big) \sim \Big(\frac{(\alpha k-d)^{1/k}}{C_k}\Big)^{\alpha  k/(\alpha k-d)} \, n^{(\alpha  \beta k -d)/(\alpha k-d)}   L\big( R_k (C_kn/m) \big)^{-d/(\alpha k-d)}. 
\end{equation}
Using \eqref{e:limit.nf} and recalling that $L$ is eventually monotone, the convergence in Cases $(i)$--$(iii)$ follows immediately.
\end{proof}
\medskip

\section{Technical Tools from Extreme Value Theory}

We present  two technical results, Potter's bounds and Vervaat's lemma. 
We refer the reader to \cite{Resnick:2007} for a detailed discussion on these results. 

\begin{lemma}[Potter's bounds from Proposition 2.6 in \cite{Resnick:2007}]  \label{l:potter}
Suppose that $U: [0,\infty)\to [0,\infty)$ is a  regularly varying function with exponent $\rho \in \R$. Given $\e > 0$,  there exists $t_0$ such that for $x \geq 1$ and $t \geq t_0$, 
$$
(1- \e) x^{\rho - \e} < \frac{U(tx)}{U(t)} < (1 + \e) x^{\rho + \e}.
$$
\end{lemma}

\begin{lemma}[Vervaat's lemma from Proposition 3.3 in \cite{Resnick:2007}]  \label{l:vervaat}
Suppose $(X_n)_{n\ge0}$ is a sequence of $\mathcal{D}[0, \infty)$-valued random elements such that $X_0$ has continuous paths. Let $e(t)$ be the identity:  $e(t)=t$.  Assume that $X_n$, $n\ge1$, has non-decreasing paths and satisfies
$$
c_n \left( X_n - e \right) \Rightarrow X_0, \ \ \text{in } \mathcal{D}[0, \infty), \ \ \ n\to\infty, 
$$
for some $c_n\to\infty$. 
It then holds that as $n\to\infty$, 
$$
c_n \left( X_n - e, X_n^{\leftarrow} - e \right) \Rightarrow (X_0, -X_0), \ \ \text{in  } \big(\mathcal D[0,\infty)\big)^2, 
$$
where $X_n^\leftarrow (s) := \inf\big\{ r: X_n(r)\ge s \big\}$ is the (left-continuous) inverse of $X_n$. 
\end{lemma}

\end{document}